\documentclass[12pt]{amsart}
\usepackage{amssymb,amsmath,amsthm,bm}

\usepackage[colorinlistoftodos,prependcaption,textsize=tiny]{todonotes}
 
\definecolor{darkblue}{RGB}{0,0,160}
 
\usepackage[lining]{libertine}   
\usepackage{cabin}
\usepackage{bbding}
\usepackage[libertine]{newtxmath}
\usepackage{amsaddr}
\usepackage[colorlinks=true,allcolors=darkblue]{hyperref}
\usepackage[backrefs,msc-links,nobysame,initials,non-compressed-cites]{amsrefs}
\usepackage{enumitem}

\usepackage[T1]{fontenc}

\usepackage{color}
\def\eps{\varepsilon}

\def\d{{\rm d}}
\def\dist{{\rm dist}}

\def\R {\mathbb{R}}
\def\N {\mathbb{N}}

\def\B {{B}}

\def\sign{{\mathrm{sign} }}

\def\Z {{\mathbb Z}}
\def \l {\langle}
\def \r {\rangle}

\def \and{\qquad\text{and}\qquad}

\newcommand{\supp}{\mathrm{supp}\,}
\newcommand{\ind}{\mathbf{1}}
 
\newcounter{thms}

\newcounter{other}
\numberwithin{other}{section}
\newtheorem{proposition}[other]{Proposition}
\newtheorem{theorem}[thms]{Theorem}
\newtheorem*{theorem*}{Theorem}
\newtheorem*{proposition*}{Proposition}
\newtheorem{corollary}{Corollary}
\numberwithin{corollary}{thms}
\newtheorem{lemma}[other]{Lemma}
\theoremstyle{definition}
\newtheorem{remark}[other]{Remark}
\newtheorem*{remark*}{Remark}
\newtheorem{definition}[other]{Definition}

\def\vint_#1{\mathchoice%
      {\mathop{\kern 0.2em\vrule width 0.6em height 0.69678ex depth -0.58065ex
              \kern -0.8em \intop}\nolimits_{\kern -0.7em#1}}%
      {\mathop{\kern 0.1em\vrule width 0.5em height 0.69678ex depth -0.60387ex
              \kern -0.6em \intop}\nolimits_{#1}}%
      {\mathop{\kern 0.1em\vrule width 0.5em height 0.69678ex depth -0.60387ex
              \kern -0.6em \intop}\nolimits_{#1}}%
      {\mathop{\kern 0.1em\vrule width 0.5em height 0.69678ex depth -0.60387ex
              \kern -0.6em \intop}\nolimits_{#1}}}
\def\vintslides_#1{\mathchoice%
      {\mathop{\kern 0.1em\vrule width 0.5em height 0.697ex depth -0.581ex
              \kern -0.6em \intop}\nolimits_{\kern -0.7em#1}}%
      {\mathop{\kern 0.1em\vrule width 0.3em height 0.697ex depth -0.604ex
              \kern -0.4em \intop}\nolimits_{#1}}%
      {\mathop{\kern 0.1em\vrule width 0.3em height 0.697ex depth -0.604ex
              \kern -0.4em \intop}\nolimits_{#1}}%
      {\mathop{\kern 0.1em\vrule width 0.3em height 0.697ex depth -0.604ex
              \kern -0.4em \intop}\nolimits_{#1}}}

\newcommand{\intav}{\vint}
\newcommand{\aveint}[2]{\mathchoice%
      {\mathop{\kern 0.2em\vrule width 0.6em height 0.69678ex depth -0.58065ex
              \kern -0.8em \intop}\nolimits_{\kern -0.7em#1}^{#2}}%
      {\mathop{\kern 0.1em\vrule width 0.5em height 0.69678ex depth -0.60387ex
              \kern -0.6em \intop}\nolimits_{#1}^{#2}}%
      {\mathop{\kern 0.1em\vrule width 0.5em height 0.69678ex depth -0.60387ex
              \kern -0.6em \intop}\nolimits_{#1}^{#2}}%
      {\mathop{\kern 0.1em\vrule width 0.5em height 0.69678ex depth -0.60387ex
              \kern -0.6em \intop}\nolimits_{#1}^{#2}}}

\renewcommand*{\cdots}{%
  \mathinner{{\cdotp}{\cdotp}{\cdotp}}%
}

\numberwithin{equation}{section}

\title[A metric approach to sparse domination]{A metric approach to sparse domination}

\author[J.\ M. Conde-Alonso]{Jos\'e M. Conde-Alonso}
\address[J.\ M. Conde-Alonso]{Universidad Aut\'onoma de Madrid - Departamento de Matem\'aticas, 7 Francisco Tom\'as y Valiente, 28049 Madrid, Spain}
\email{\href{mailto:jose.conde@uam.es}{\textnormal{jose.conde@uam.es}}}
\thanks{J.\ M.\ Conde-Alonso was partially  supported by ERC Grant 32501 and has been partially supported by the Madrid Government (Comunidad de Madrid-Spain) under the Multiannual Agreement with Universidad Autónoma de Madrid in the line of action encouraging youth research doctors, in the context of the V PRICIT, grant number SI1/PJI/2019-00514.}

\author[F.\ Di Plinio]{Francesco Di Plinio} 
\address[F.\ Di Plinio]{Department of Mathematics, Washington University in Saint Louis,  1 Brookings Drive, Saint Louis, Mo 63130, USA}
\email{\href{mailto:francesco.diplinio@wustl.edu}{\textnormal{francesco.diplinio@wustl.edu}}\Envelope} 
\thanks{F. Di Plinio was partially supported by the National Science Foundation under the grants     NSF-DMS-1650810, NSF-DMS-1800628, NSF-DMS-2000510, by the Severo Ochoa Program SEV-2013-0323 and by Basque Government BERC Program 2014-2017}

\author[I. Parissis]{Ioannis Parissis}
  \address[I. Parissis]{Departamento de Matem\'aticas, Universidad del Pais Vasco, Aptdo. 644, 48080 Bilbao, Spain  
\newline  \indent and  Ikerbasque, Basque Foundation for Science, Bilbao, Spain}
\email{\href{mailto:ioannis.parissis@ehu.es}{\textnormal{ioannis.parissis@ehu.es}}}
\thanks{I. Parissis is partially supported by the project PGC2018-094528-B-I00 (AEI/FEDER, UE) with acronym ``IHAIP'', grant T1247-19 of the Basque Government and IKERBASQUE.}

\author[M.\ N. Vempati]{Manasa N. Vempati} 
\address[M.\ N.\ Vempati]{Department of Mathematics, Washington University in Saint Louis,   1 Brookings Drive, Saint Louis, Mo 63130, USA}
\email{\href{mailto:m.vempati@wustl.edu}{\textnormal{m.vempati@wustl.edu}}}

\subjclass[2010]{Primary: 42B20. Secondary: 42B25}
\keywords{Sparse domination,  quasi-metric spaces, singular integral operators, Radon transform, weighted norm inequalities}

\begin{document}

%%%%%%%%%%%%%%%%%%%%%%%%%%%%%% ABSTRACT ABSTRACT ABSTRACT
\begin{abstract}
We present a general approach to sparse domination based on single-scale   $L^p$-improving  as a key assumption. The results are formulated in the setting of metric spaces of homogeneous type and avoid completely the use of dyadic-probabilistic techniques as well as of Christ-Hyt\"onen-Kairema cubes. Among the applications of our general principle, we recover  sparse domination of Dini-continuous Calder\'on-Zygmund kernels on spaces of homogeneous type, we prove   a family of sparse bounds for maximal functions associated to  convolutions with measures exhibiting Fourier decay, and we deduce sparse estimates for Radon transforms along polynomial submanifolds of $\R^n$. \end{abstract}
%%%%%%%%%%%%%%%%%%%%%%%%%%%%%% ABSTRACT ABSTRACT ABSTRACT

\maketitle

%%%%%%%%%%%%%%%%%%%%%%%%%%%%%% SECTION SECTION SECTION
\section{Introduction}  The prototypical example of a singular integral operator of interest in Harmonic Analysis, the Hilbert transform, may be decomposed into the $\ell^1$-superposition over scales of convolutions  with   a  suitably chosen and rescaled smooth function. This paradigm of superposition of single-scale operators is most general, and also extends in particular to Radon transforms, defined by convolution with a measure supported on a lower dimensional set. 

Sparse domination theory rose to prominence in the pursuit of sharp weighted norm inequalities for Calder\'on-Zygmund operators, through the seminal works of Lerner \cite{Ler1,Ler2015,Ler2016}, and Lacey \cite{Lac2017}. Its main thrust is to estimate, pointwise, in dual form or in norm, the singular integral $Tf$ by a sparse operator. That is, a tamer, positive and localized multiscale operator $Sf$ which is a superposition of averages of $f$ on a \emph{sparse}-- i.e.\ having pairwise disjoint major subsets--  collection of cubes. This control is performed via some type of high-low, localized  cancellation enjoyed by $T$, and has since been carried out for much more singular operators than those of Cald\'eron-Zygmund type: a non-exhaustive list includes  modulation invariant operators \cite{CuDPOu}, non-integral operators \cite{BFP}, rough kernels \cite{CoCuDPOu,HRT} oscillatory integrals \cite{KL}, discrete singular integrals \cite{CKL,HKLY}, and, most importantly for the present article, Radon transforms, beginning with the work of Lacey \cite{LSMF} on the spherical maximal operator
\[
A f(x) = \sup_{t>0} |A_t f(x)|, \qquad x\in \R^d,
\]
where $A_t$ is the spherical average on $x+ t \mathbb S^{d-1}$. Lacey showed that a high-low cancellation type  strengthening of the well known $L^p$-improving property  of  the single scale operator $f\mapsto \sup_{t\sim 1 } |A_tf|$ may be upgraded into a sparse domination type result by a slick reformulation of the high-low scheme employed in \cite{CoCuDPOu}. The work \cite{LSMF} was followed by the moment curve analogue of Cladek and Ou \cite{ClOu}, see also \cite{Ob}, and by the general result of Hu \cite{Hu}, which achieves a sparse domination type result, and consequent weighted norm inequalities, for singular integrals on finite type submanifolds in the generality of Christ, Nagel, Stein and Wainger \cite{CNSW}. The works \cites{ClOu,Ob,Hu} operate at different levels of generality within the footprint of \cite{LSMF}: in particular,   the integral representation of the kernel is relied upon at different points, and the iterative argument leading to sparse domination involves a discretization of the operator which is made possible by variants of the Christ-Hyt\"onen-Kairema dyadic systems in spaces of homogeneous type \cite{MCh,HKa}. In fact, the dominating sparse operator involves averages over these dyadic cubes.

This article develops  a  sparse domination principle, formulated in the generality of homogeneous measure metric spaces \emph{under natural minimal structural assumptions}. To begin with,   the operators of interest can be written as sums of single-scale operators, and each such piece is well localized. No integral representation for our operators is assumed, in particular the kernel estimates of Calder\'on-Zygmund theory are not generally available. Instead, these are replaced by postulating a version of suitably normalized $L^p$-improving estimates for each single-scale piece of the operator, in accordance with the  approach of Lacey in \cite{LSMF}, then leading to a corresponding sparse domination result for both $\ell^1$-multiscale, and maximal operators of this nature. The dominating sparse form involves averaging over sparse collections of quasimetric balls. No appeal to dyadic systems in the vein of \cite{MCh,HKa} is needed. Our results appear to be optimal in this  sense: if   a sparse domination holds result in an open range of indices, then the scale-invariant $L^p$-improving property follows. 

The main results of this article, Theorem \ref{th:SD} for the $\ell^1$ sum and Theorem \ref{th:max} for the maximal operator associated to a sequence of single scale operator $T(s)$, may be in fact  more detailedly described as follows. In the general context of spaces of homogeneous type,  in addition to the structural single scale localization property of each $T(s)$, it suffices to have uniform $L^p$-boundedness of partial sums or of the maximal operator and an $L^{p_1}\to L^{p_2 {'} }$-improving property with modulus of continuity to ensure a $(p_1,p_2)$ sparse bound. In fact we work with Dini-type moduli of continuity, in line with the best known results in the Euclidean case \cite{CoCuDPOu,HRT,Ler2016}, and an improvement over previous work in the context of Radon transforms and maximal Fourier multipliers \cite{ClOu,LSMF,Ob,Hu}. See also  for instance \cite{KLi1} for sparse domination theorems within the context of Calder\'on-Zygmund operators with rougher moduli of continuity.
 Also, Proposition \ref{prop:partial} provides a converse to Theorems \ref{th:SD}, \ref{th:max} in an open range of exponents. The main results are stated in Section \ref{s:mr}, together with laying out the framework of spaces of homogeneous type we work with. 

Section \ref{sec:Lpimproving} details several concrete applications of our main theorems. The first two are of classical nature: Theorem \ref{thm:czo} is a new form of the well-known sparse domination for Dini-continuous Calder\'on-Zygmund operators on spaces of homogeneous type. This result  was first obtained in \cite{Kara,VZK},  extending to Dini moduli of continuity the $A_2$ theorem in homogeneous spaces of \cite{NRV}. Unlike \cites{Kara,VZK,NRV}, our proof does not rely in any way on  the  dyadic systems constructed in \cite{MCh,HKa}.  Corollary \ref{c:geomax} is instead a deduction of the sparse bound for  geometric maximal operators in spaces of homogeneous type.

The most conspicuous applications are provided in the context of maximal and singular Radon transforms. Our general setting is the Euclidean  space $\R^n$ equipped with a quasi-norm which is homogeneous with respect to a dilation semigroup $\{\delta_r:r>0\}$.    Theorem \ref{thm:sparseFT}, 	which is suitably deduced from Theorem \ref{th:max},  contains  a sparse estimate for the maximal operator 
\[
T_{\star } f \coloneqq \sup_{s} |f*\d\mu_s| 
\]
where each $\mu_s$ is the $\delta_s$-pushforward of some Borel measure $m^s$ supported in the unit metric ball, provided that the $m^s$ have uniform algebraic Fourier decay rate at $\infty$. An analogous estimate is established for the $\ell^1$-sum if $m^s$ are of  cancellative nature, using Theorem \ref{th:SD} instead. Theorem \ref{thm:sparseFT} is a sparse version of the stalwart result of Duoandikoetxea and Rubio de Francia \cite[Theorems A and B]{DR}, and it has not appeared in previous literature. As a further application, we derive from it a sparse domination theorem for singular Radon transform along polynomial subvarieties of $\R^n$, stated in Corollary \ref{c:radon1}. Results of this type for the $\ell^1$-sum are contained in the recent article by Hu \cite{Hu}, in fact within the  more general framework of singular integrals on finite-type submanifolds in the vein of \cite{CNSW}. While it is plausible that the arguments of  \cite{Hu} may be adapted to cover the maximal function case, the maximal case of Corollary \ref{c:radon1} has not appeared before.

The structure of the remaining sections of the paper is as follows. Section \ref{sec2} contains some necessary preliminaries about Whitney coverings in geometrically doubling metric spaces, which are relevant in the proof of the main theorems. In fact, part of the interest of this paper comes from demonstrating that the  Whitney covering properties are sufficient to generate a  sparse collection of balls whose associated sparse $(p_1,p_2)$-form controls all localized and improving operators. Such collection is constructed in Lemma \ref{lem:tedious} and subsequently employed in   the proofs of Theorems \ref{th:SD} and \ref{th:max}, which   are carried out in Section \ref{sec:proofmain} to \ref{sec:max}. 

%%%%%%%%%%%%%%%%%%%%%%%%%%%%%% REMARK REMARK REMARK
\begin{remark*}Before the first version of this article was made publicly available, David Beltran, Joris Roos and Andreas Seeger kindly shared with us their  preprint  \cite{BRS} on multi-scale sparse domination. Although in different settings,  both papers use scale-invariant versions of the $L^p$-improving property as a standing assumption and cover some classes of singular Radon transforms such as the one in \cite{Ob}. We thank the authors of  \cite{BRS} for sharing their preprint.
\end{remark*}
%%%%%%%%%%%%%%%%%%%%%%%%%%%%%% REMARK REMARK REMARK

%%%%%%%%%%%%%%%%%%%%%%%%%%%%%% SECTION SECTION SECTION
\subsection*{Acknowledgments} This research project originated  during the workshop on \emph{Sparse Domination of Singular Integrals}, held at the American Institute of Mathematics (AIM), October 9--13, 2017. The authors want to express their gratitude to the personnel and staff of AIM. The authors would like to thank the referees for an expert reading and several helpful comments, leading to a significant improvement of the article.

%%%%%%%%%%%%%%%%%%%%%%%%%%%%%% SECTION SECTION SECTION
\section{Preliminaries and main results} \label{s:mr} The next paragraphs describe our general setup and are instrumental to the  forthcoming statement of the main results. 

%%%%%%%%%%%%%%%%%%%%%%%%%%%%%% SECTION SECTION SECTION
\subsection{The space of homogeneous type} Let ${\mathbb{X}}$ be a set equipped with a quasi-metric $\d$. Here and in what follows a function $\d:{\mathbb{X}} \times {\mathbb{X}} \to [0,\infty)$ will be called a \emph{quasi-metric} if for all $x,y,z\in {\mathbb{X}}$ we have that $\d(x,y)=0\Leftrightarrow x=y$ and there exist constants $c_\d,\tilde c_\d ' \geq 1$ such that $\d(x,y)\leq \tilde c_\d \d(y,x)$ and $\d(x,y)\leq c_\d(\d(x,z)+\d(z,y))$. In order to simplify our notation we will assume in what follows that $\d$ is symmetric and satisfies a quasi-triangle inequality with a possibly larger constant $c_\d.$ We can easily achieve that by symmetrizing $\d(x,y)$; thus we will have
\[
\d(x,y)=\d(y,x),\qquad \d(x,y)\leq c_\d [\d(x,z)+\d(z,y)],\qquad x,y,z\in\mathbb X.
\]
We denote by 
\[
B(x,r)\coloneqq \{y\in {\mathbb{X}}:\, \d(x,y)<r\}
\]
the quasi-metric ball with radius $r$, centered at $x$, and assume without loss of generality that the balls $B(x,r)$ are open in the sense that for all $x\in {\mathbb{X}},r>0$ and  $x'\in B(x,r)$ there exists $r'>0$ such that $B(x',r')\subset B(x,r)$. Throughout this paper, if $B=B(x,r)$, we denote by $\alpha B$ the ball with same center and $\alpha$-times the radius, namely $\alpha B\coloneqq B(x,\alpha r)$. We will assume that each ball $B$ in ${\mathbb{X}}$ comes with a fixed center $c_B$, and radius $r_B$, although these are not necessarily uniquely determined by $B$. We say that a Borel measure $|\cdot|$ on ${\mathbb{X}}$ is $(\alpha,\beta)$-\emph{doubling} for some $\alpha,\beta> 1$ if for all balls $B$
\[
|\alpha B| \leq \beta  |B|.
\]  
If $|\cdot|$ is $(\alpha,\beta)$-\emph{doubling} for some $\alpha,\beta> 1$ then we refer to the triple $({\mathbb{X}},\rho,|\cdot|)$ as a \emph{space of homogeneous type}. We simply write $L^p$ for $L^p({\mathbb{X}},\rho,|\cdot|)$ and of course all $\mathrm{d} x$ integrations that appear in this paper are with respect to $|\cdot|$. Throughout the paper we will write $L^p$-averages, with respect to the underlying measure and some (metric) ball $B$ as
\[
\l f\r_{p,B}\coloneqq \bigg(\frac{1}{|B|}\int_B |f|^p\bigg)^\frac{1}{p}\eqqcolon \bigg(\intav_B |f|^p\bigg)^\frac{1}{p}.
\]
Without loss of generality we can and will assume throughout the paper that all our doubling measures are $(2,\beta)$-doubling for some $\beta>1$. We will always assume that $\mathbb{X}$ is geometrically doubling, see Definition~\ref{def:geomdoub} and the discussion in \S\ref{def:geomdoub} for further details.

%%%%%%%%%%%%%%%%%%%%%%%%%%%%%% SECTION SECTION SECTION
\subsection{An operator which is a sum of single scale pieces}\label{s.setup} This paragraph  details the environment and standing assumptions for our main results. Let $({\mathbb{X}},\d,|\cdot|)$ be a fixed space of homogeneous type and write $\mathrm{Lip}({\mathbb{X}})$ for the Lipschitz functions on $\mathbb X$. Consider a linear operator $T$, initially defined on all $f\in \mathrm{Lip}({\mathbb{X}})$ with compact support, and assume that $T$ can be written formally as a sum $T\sim \sum_{s\in \Z} T(s) $ where each $T(s)$ is a linear operator. The reader is encouraged to think of $T(s)$ as being a --possibly singular-- average at scale $2^s$.  For $\sigma,\tau\in\Z$   set
\[
T_\sigma ^\tau f(x)\coloneqq \sum_{ \sigma \leq s<  \tau}  [T(s) f](x),\qquad x\in\mathbb X,
\]
with the understanding that $T_{\sigma} ^\tau f\equiv 0$ if $\sigma\geq \tau$. We assume that $\d$ localizes the operators $T_{\sigma} ^\tau$ in this sense: there exists a constant $c_o\geq 1$ such that for all balls $L$ in ${\mathbb{X}}$ with $r_L=2^{s_L}$ and $\sigma \leq s\leq s_L$, there holds
\begin{equation} \label{eq:localization}
\supp ({T_\sigma ^{s}} [f\ind_L]) \subset c_o L.
\end{equation}
We make the quantitative assumption that
\begin{equation}\label{eq:unifbdd}
\sup_{\sigma<\tau}  \| T_{\sigma}^\tau \|_{p\to p} \eqqcolon C_p <\infty
\end{equation}
for some $1<p<\infty$.

%%%%%%%%%%%%%%%%%%%%%%%%%%%%%% REMARK REMARK REMARK
\begin{remark}\label{rmrk:represent} In this abstract setup the operator $T$ is not currently well defined which is why the vague notation $T\sim \sum_{s}T(s)$ is used. The question of whether and how the infinite sum converges to $T$ is unspecified. In applications, we will actually start with a  concrete operator $T$, discretize it as a sum over scales $\sum_s T(s)$, and try to recover $T$ as a weak limit of the truncated sums. Typically what happens is that the weak limit of the truncated sums can differ from the original operator $T$ by a pointwise multiplication operator; see \cite[\S I.7.2]{stein}. Thus, for the purposes of formulating an abstract theorem we will assume the uniform bound \eqref{eq:unifbdd} for the truncations and realize $T$ in the form
\begin{equation}\label{eq:T}
\l Tf,g\r = \l T_0f,g\r +\l mf,g\r,
\end{equation}
where $m\in L^\infty(\mathbb X)$ and there exist sequences $\sigma_j\to -\infty,\tau_k\to +\infty$ such that for all $f,g\in  \mathrm{Lip}({\mathbb{X}})$ with compact support we have
\[
\lim_{j,k} \l T_{\sigma_j} ^{\tau_k} f,g\r = \l T_0 f,g\r.
\]
Thus, up to taking subsequences, we can always think of $\hspace{.1em}T$ as being the weak limit of the truncations $T_{\sigma} ^\tau$ modulo a pointwise multiplication operator by a bounded function $m$ and in view of \eqref{eq:T} we will have $\|T_0\|_{p\to p}+\|m\|_\infty\lesssim C_p$. We will come back to that point in the proof of the main theorem, Theorem~\ref{th:SD}, in \S\ref{s.together}.
\end{remark}
%%%%%%%%%%%%%%%%%%%%%%%%%%%%%% REMARK REMARK REMARK
 
Next is the reformulation of the  $(p_1,p_2')$-improving assumption  in the context of metric measure spaces of homogeneous type. For this formulation it is convenient to introduce the concept of $(p,r)$-atoms.

%%%%%%%%%%%%%%%%%%%%%%%%%%%%% DEFINITION DEFINITION DEFINITION
\begin{definition}\label{d.molecule} Let $p\geq 1, r>0$, and $B\subset \mathbb X\,$ be a ball  of radius $r$. We say that $b\in L^p(\mathbb X)$ is a \emph{$(p,r)$-atom supported on $B$},  if \[
\mathrm{supp}\, b\subset B,\qquad \int b=0.\]  
\end{definition}
%%%%%%%%%%%%%%%%%%%%%%%%%%%%%% DEFINITION DEFINITION DEFINITION

The term \emph{modulus of continuity} will refer to an increasing continuous function $\omega:[0,1)\to[0,\infty)$ such that $\lim_{t\to 0^+}\omega(t)=\omega(0)=0$. In the context of singular integral operators, a prominent role is played by  the \emph{Dini moduli of continuity} satisfying 
\[
\|\omega\|_{\mathrm{Dini}}\coloneqq\int_0^1  \omega( \delta) \, \frac{\d \delta}{\delta} <\infty.
\]

%%%%%%%%%%%%%%%%%%%%%%%%%%%%%% DEFINITION DEFINITION DEFINITION
\begin{definition}\label{d.LPimprove} Let $p_1,p_2\in[1,\infty]$ with $p_2 '\geq p_1$ and $s\in\Z$. We say that $T\sim\sum_\ell T(\ell)$ is $(p_1,p_2 ')$-improving  at scale $s$ with constant $I_{p_1,p_2}$ and modulus $\omega$ if there exist constants $\gamma_1,\gamma_2\geq 1$ such that: \vskip1mm
\noindent a.  
 for all $f\in L^{\infty}(\mathbb X)$ and all balls $L$ with radius $2^{s}\leq r_L \leq \gamma_1 2^s$, we have
\begin{equation}\label{eq:Lpimpss}
\left\langle T(s)   (f\ind_L)  \right\rangle_{p_2',\gamma_2 L} \leq I_{p_1,p_2}    \l f \r_{p_1, L}; 
  \end{equation}
  \vskip1mm
\noindent b.  for all $f\in L^{\infty}(\mathbb X)$, all balls $L$ with radius $2^{s}\leq r_L \leq \gamma_1 2^s$, and all $(p_2,r)$ atoms $b$ with    $r\leq 2^s$, there holds
\begin{equation}\label{Lpimp}
|\l T(s)   (f\ind_L), b  \r| \leq\omega \left({\textstyle\frac{r}{2^s}}\right) |L|  \l f \r_{p_1, L}   \l b  \r_{p_2,\gamma_2 L}.
\end{equation} 
\end{definition}
%%%%%%%%%%%%%%%%%%%%%%%%%%%%%% DEFINITION DEFINITION DEFINITION

%%%%%%%%%%%%%%%%%%%%%%%%%%%%%% SECTION SECTION SECTION
\subsection{Main results}\label{sec:mainres} Given $\zeta\in(0,1)$, say that the collection of measurable sets $\mathcal{A}$ is \emph{$\zeta$-sparse} if it is countable and for all $A \in \mathcal A$ there exists a set $E_A\subset A$ with \[|E_A|>\zeta |A|,  \qquad A, A'\in \mathcal A, \, E_A \cap E_{A'}\neq \varnothing \implies A=A'.\]  A collection $\mathcal A$ is called \emph{sparse} if it is $\zeta$-sparse for some fixed $\zeta\in(0,1)$. In words, a collection of sets is sparse if it is pairwise disjoint up to passing to a major subset.
Sparse collections are featured in the first main result of the article.

%%%%%%%%%%%%%%%%%%%%%%%%%%%%%% THEOREM THEOREM THEOREM
\begin{theorem} \label{th:SD} Let $({\mathbb{X}},\d,|\cdot|)$ be a space of homogeneous type. Let $1\leq p_1\leq p_2 '\leq \infty$,  $\omega$ be a  Dini modulus of continuity and let $T$ be a linear operator on $({\mathbb{X}},\d,|\cdot|)$ satisfying structural assumption  \eqref{eq:localization}. Furthermore, assume:
\begin{itemize}
\item[1.] estimate
 \eqref{eq:unifbdd} holds for some  $1<p<\infty$ with constant $C_{p}$;
\item[2.] $T$ is $(p_1,p_2')$-improving at every scale $s\in\Z$ with uniform constant $I$ and    modulus $\omega$; 
\item[3.] $T^*$ is $(p_2,p_1')$-improving  at every scale $s\in\Z$ with uniform  constant $I^*$ and    modulus $\omega$. 
\end{itemize}
 Then, for all $f_1,f_2\in \mathrm{Lip}({\mathbb{X}})$ with compact support and every $\sigma,\tau\in\Z$ with $\sigma<\tau$ there exists a sparse collection $\mathcal S_{\sigma,\tau}$ consisting of $\d$-balls $B$ with $2^{\sigma}\leq r_B \leq 2^{\tau}$ such that
\[
|\l T_{\sigma} ^\tau f_1, f_2 \r| \lesssim \left(C_{p}+ I + I^*  +\|\omega\|_{\mathrm{Dini}}\right) \sum_{B \in \mathcal S_{\sigma,\tau}} |B| \l f_1 \r_{p_1,c_1B}  \l f_2 \r_{p_2, c_1B},
\]
where $c_1$ is a fixed constant depending on the homogeneous structure of $\,\mathbb X$ and $c_o$ of \eqref{eq:localization}.

Furthermore if $\,T$ is defined through \eqref{eq:T}  then for all $f_1,f_2\in \mathrm{Lip}({\mathbb{X}})$  there exists a sparse collection $\mathcal B$ consisting of $\d$-balls such that
\[
|\l Tf_1, f_2 \r| \lesssim \left(C_{p}+ I + I^*  +\|\omega\|_{\mathrm{Dini}}\right) \sum_{B \in \mathcal S} |B| \l f_1 \r_{p_1,B}  \l f_2 \r_{p_2, B}.
\]
The implicit constants depend on the homogeneous metric structure of $({\mathbb{X}},\d,|\cdot|)$ and the constant in \eqref{eq:localization} but are independent of $\sigma,\tau,f_1,f_2$.
\end{theorem}
%%%%%%%%%%%%%%%%%%%%%%%%%%%%%% THEOREM THEOREM THEOREM

%%%%%%%%%%%%%%%%%%%%%%%%%%%%%% SECTION SECTION SECTION
\subsubsection{A maximal version} Below follows a variation of Theorem~\ref{th:SD}  providing a sparse domination result for abstract maximal operators in metric spaces of homogeneous type. In order to set it up we consider again an abstract sequence of linear operators $\{T(s)\}_{s\in \Z}$. Assume the localization condition: there exists a constant $c_0\geq 1$ such that for every metric ball $L$ with $r_L=2^{s_L}$
\begin{equation}\label{eq:localmax}
\mathrm{supp}(T(s)[f\ind_L])\subset c_o L\qquad \forall  s\leq s_L.
\end{equation}
We consider the maximal operator
\[
T_\star f(x)\coloneqq \sup_{s\in\Z} |T(s)f(x)|,\qquad x\in\mathbb X.
\]

%%%%%%%%%%%%%%%%%%%%%%%%%%%%%% THEOREM THEOREM THEOREM
\begin{theorem} \label{th:max} Let  $1\leq p_1\leq p_2' \leq \infty$,  $\{T(s)\}_{s\in\Z}$ be a sequence of linear operators satisfying \eqref{eq:localmax} and such that $T_\star$ is bounded on $L^\infty(\mathbb X)$. Assume that  for each $s\in\Z$ the operator $T^*(s)$ is $(p_2,p_1')$-improving at scale $s$ with  constant $I^*$  uniformly in $s$ and  a Dini modulus of continuity $\omega$. Then, for all $f_1,f_2$ bounded functions with compact support  and $\sigma,\tau\in\Z$ with $\sigma<\tau$ there exists a sparse collection $\mathcal S_{\sigma,\tau}$ consisting of $\d$-balls $B$ with $2^\sigma \leq r_B\leq 2^\tau$ such that
\[
\left|\left\langle \sup_{\sigma\leq s< \tau}|T(s) f_1|, f_2 \right\rangle\right| \lesssim \left(\|T_\star\|_{L^\infty(\mathbb X)}+I^* +\|\omega\|_{\mathrm{Dini}}\right) \sum_{B \in \mathcal S_{\sigma,\tau}} |B| \l f_1 \r_{p_1,c_1B}  \l f_2 \r_{p_2, c_1B},
\]
where $c_1$ is a fixed constant depending on the homogeneous structure of $\,\mathbb X$ and $c_o$ of \eqref{eq:localization}.

Furthermore for all $f_1,f_2$ bounded functions with compact support there exists a sparse collection $\mathcal S$ consisting of $\d$-balls such that
\[
|\l T_\star f_1, f_2 \r| \lesssim \left(\|T_\star\|_{L^\infty(\mathbb X)}+I^* +\|\omega\|_{\mathrm{Dini}}\right)  \sum_{B \in \mathcal S } |B| \l f_1 \r_{p_1,B}  \l f_2 \r_{p_2, B}.
\]
The implicit constants above depend on the homogeneous metric structure of $({\mathbb{X}},\d,|\cdot|)$ and the constant in \eqref{eq:localmax}, but are independent of $\sigma,\tau,f_1,f_2$.
\end{theorem}
%%%%%%%%%%%%%%%%%%%%%%%%%%%%%% THEOREM THEOREM THEOREM

Theorems \ref{th:SD} and \ref{th:max} above have a partial converse which in several concrete realizations becomes a full converse. For the abstract setup we content ourselves with stating the following proposition with a stronger statement coming up in Lemma~\ref{lem:thenewlemma} of the next section.

%%%%%%%%%%%%%%%%%%%%%%%%%%%%%% THEOREM THEOREM THEOREM
\begin{proposition}\label{prop:partial} Suppose that for each $\sigma,\tau\in\Z$ with $\sigma< \tau$ the operator $T\sim \sum_s T(s)$ satisfies the localization property \eqref{eq:localization} and the conclusion of Theorem~\ref{th:SD} or \eqref{eq:localmax} and the conclusion of Theorem~\ref{th:max}. Then for every ball $L$ with $r_L=2^s$ we have
	\[
	\l T(s)(f\ind_L)\r_{p_2 ',L} \leq \l f\r_{p_1,L}
	\]
with implicit constant depending on the localization properties of $\hspace{.2em}T$ and the constants in the sparse domination assumption but not on $L$ or $s$.
\end{proposition}
%%%%%%%%%%%%%%%%%%%%%%%%%%%%%% THEOREM THEOREM THEOREM

%%%%%%%%%%%%%%%%%%%%%%%%%%%%%% PROOF PROOF PROOF
\begin{proof}	 Using the existence of a sparse bound in the form of either Theorem~\ref{th:SD} or Theorem~\ref{th:max} we conclude that for each $s\in \Z$ there is a sparse collection $\mathcal B$ consisting of balls of radius $2^s$ such that
	\[
	|\l T(s)(f\ind_L),g\r|\lesssim \sum_{B\in\mathcal B} |E_B|\l f\ind_L\r_{p_1,B}\l g\ind_{c_oL}\r_{p_2,B}
	\]
with $\{E_B\}_{B\in\mathcal B}$ disjoint. The above estimate and the doubling assumptions on $(\mathbb X,\d,|\cdot|)$ then imply that
	\[
	|\l T(s)(f\ind_L),g\r|\lesssim  |L|\l f\ind_L\r_{p_1,L}\l g\ind_{c_oL}\r_{p_2,L}
	\]
which by duality yields the desired conclusion.
\end{proof}
%%%%%%%%%%%%%%%%%%%%%%%%%%%%%% PROOF PROOF PROOF

Concrete realizations of Theorem \ref{th:SD} and Theorem~\ref{th:max} are postponed to Section \ref{sec:Lpimproving}, where the $(p_1,p_2')$ improving condition is suitably reinterpreted in a more familiar form. Here, we point out that Theorem \ref{th:SD} yields as a corollary quantitative weighted norm inequalities of $A_{p}\cap RH_q$ type for the operator $T$. This theme has recently been pursued for several classes of operators within and beyond the scope of Calder\'on-Zygmund theory in the Euclidean setting, see for instance \cite{BFP,CoCuDPOu,CuDPOu,LSMF,Ler2015}.

We briefly recall the definition of $A_p$ weights in the context of quasimetric measure space of homogeneous type $(\mathbb X,\d,|\cdot|)$.  These are locally integrable non-negative functions $w$ such that
\[
[w]_{A_p}\coloneqq \sup_B \l w\r_{B}\l w^{-\frac{1}{p-1}}\r_{B} ^{p-1}<\infty,\qquad 1<p<\infty,
\]
where the supremum is taken over all $\d$-balls and all the integrations are with respect to the doubling measure $|\cdot|$. For $p=1$, define $[w]_{A_1}$ to be the smallest constant $c>0$ such that for all metric balls $B$ 
\[
\l w\r_{1,B}\leq c \inf_B w
\]
The \emph{Reverse H\"older class} $\mathrm{RH}_p$ is defined for $1< p <\infty$ as the class of non-negative locally integrable functions $w$ on $\mathbb X$ such that
\[
[w]_{\mathrm{RH}_p}\coloneqq \sup_B \frac{\l w\r_{p,B}}{\l w\r_{1,B}}<\infty.
\]
The proof of the following weighted estimate is an easy consequence of the sparse domination result of Theorem~\ref{th:SD}; see for example \cite[\S6]{BFP} or \cite{CDPOUBP}. A similar corollary holds for the maximal version $T_\star(f)\coloneqq \sup_{s\in\Z}|T(s)f|$; we omit the details.
%%%%%%%%%%%%%%%%%%%%%%%%%%%%%% COROLLARY COROLLARY COROLLARY
\begin{corollary} Let $T\sim \sum_{s\in\Z}T(s)$ in the sense of \eqref{eq:T} and assume that $T$ satisfies the assumptions of Theorem~\ref{th:SD} for $1\leq p_1< p_2'<\infty$. Then for any $p_1<p<p_2'$ and $w\in A_{p/p_1}\cap \mathrm{RH}_{p_2'/p}$ we have the following weighted norm estimate
\[
\|T:L^p(w)\to L^p(w)\| \lesssim \bigg([w]_{A_{\frac{p}{p_1}}}[w]_{\mathrm{RH}_{\big(\frac{ p' _2 }{p}\big)'}}\bigg)^{\max\big(\frac{1}{p-p_1},\frac{p_2' -1}{p_2 ' - p}\big)}.
\]
The implicit constant depends on the assumptions for $T$, the homogeneous structure of $(\mathbb X,\d,|\cdot|)$ and the indices $p_1,p_2,p$.
\end{corollary}
%%%%%%%%%%%%%%%%%%%%%%%%%%%%%% COROLLARY COROLLARY COROLLARY

%%%%%%%%%%%%%%%%%%%%%%%%%%%%%% SECTION SECTION SECTION
\section{Applications: the $L^p$-improving property revisited} \label{sec:Lpimproving}

%%%%%%%%%%%%%%%%%%%%%%%%%%%%%% SECTION SECTION SECTION
\subsection{Calder\'on-Zygmund theory} This subsection is a digression devoted to the  description of  classical Calder\'on-Zygmund operators in the homogeneous setup. These are themselves $L^p$-improving operators \emph{par excellence} and help illustrate  and contextualize the definitions in this paper. We make this precise below.

%%%%%%%%%%%%%%%%%%%%%%%%%%%%%% SECTION SECTION SECTION
\subsubsection{Calder\'on-Zygmund operators} We begin by giving the formal definition of Calder\'on-Zygmund operator on a space of homogeneous type $(X,\d,|\cdot|)$. The definition below  is given for Dini continuous operators, but of course more general definitions are possible.
%%%%%%%%%%%%%%%%%%%%%%%%%%%%%% DEFINITION DEFINITION DEFINITION
\begin{definition}\label{def:CZO} Let $(\mathbb X,\d,|\cdot|)$ be a space of homogeneous type. We say that $T$ is a Dini-Calder\'on--Zygmund operator on $\mathbb X$ if $T$ is bounded on $L^p(\mathbb X)$ for some $p\in(1,\infty)$ and there exists a kernel $K:\mathbb X\times \mathbb X\setminus\{x,y\in\mathbb X:\, x=y\}\to \mathbb C$ such that for all $f\in  \mathrm{Lip}({\mathbb{X}})$ with compact support 
\[
T(f)(x)= \int_{X} K(x,y)f(y)\, \d y ,\qquad \forall x\notin {\rm supp} f.
\]
The kernel $K(x,y)$ is assumed to satisfy the following size and regularity conditions: there exist constants $C_T,A>1$ such that for all $x\not= y$,
\begin{equation}\label{eq:CZOsize}
    |K(x, y)| \leq {\frac{C_T}{V(x, y)}},
\end{equation}
and for pairwise different $x,x',y\in\mathbb X$ with $\d(x, x')\leq A^{-1} \d(x, y)$ there holds 
\begin{equation}\label{eq:CZOreg}
    |  K(x, y) - K(x', y) |+|  K(y,x) - K(y,x') |  \leq \frac{C_T}{V(x,y)}\omega\left({\frac{d(x, x')}{ d(x,y)}}\right).
\end{equation}
In the estimates above   $V(x,y)\coloneqq |B(x,\d(x,y))|$, and $\omega$ is a \emph{Dini} modulus of continuity. Note that by the doubling condition on the measure $|\cdot|$ we have that $V(x,y)\simeq V(y,x)$.
\end{definition}
%%%%%%%%%%%%%%%%%%%%%%%%%%%%%% DEFINITION DEFINITION DEFINITION
We first recall an easy decomposition of Calder\'on-Zygmund operators into local pieces. Given $T$ a Calder\'on-Zygmund operator on $(\mathbb X,\d,|\cdot|)$ associated with a kernel $K$ we define
\[
T(f)(x)=\sum_{s\in \mathbb Z} [T(s)f](x)\coloneqq \sum_{s\in\mathbb Z} \int_{2^s\leq \d(x,y)<2^{s+1}} K(x,y) f(y)\,\d y,\qquad x\in \mathbb X.
\]
We set
	\[
	K_s(x,y)\coloneqq K(x,y)\ind_{\{(x,y)\in\mathbb X \times \mathbb X:\, 2^s\leq \d(x,y)<2^{s+1}\}}(x,y)
	\]
so that for $s\in \Z$
\[
T(s)f(x)=\int K_s(x,y) f(y)\,\d y,\qquad x\in \mathbb X.
\]
Note that the formula above makes sense for functions $f$ which are Lipschitz with compact support as we restrict $(x,y)$ away from the diagonal and $K$ satisfies the size condition \eqref{eq:CZOsize}. Furthermore, it is well known that  maximal truncations of Calder\'on-Zygmund operators on metric spaces of homogeneous type are uniformly $L^p(\mathbb X)$-bounded, see for example~\cite[\S I.7]{stein}.

With these definitions in hand, it is easily verified that the truncations of $\hspace{.1em}T$ satisfy the localization properties \eqref{eq:localization}. Furthermore, we can readily check that $T$ is $(1,\infty)$ improving.
%%%%%%%%%%%%%%%%%%%%%%%%%%%%%% LEMMA LEMMA LEMMA
\begin{lemma}\label{lem:imprczo} Let $T=\sum_{s}T(s)$ be a Dini-Calder\'on-Zygmund operator on $(\mathbb X,\d,|\cdot|)$ as  defined above. Then $T$ and $T^*$ are $(1,\infty)$ improving in the sense of Definition~\ref{d.LPimprove}.	
\end{lemma}
%%%%%%%%%%%%%%%%%%%%%%%%%%%%%% LEMMA LEMMA LEMMA

%%%%%%%%%%%%%%%%%%%%%%%%%%%%%% PROOF PROOF PROOF
\begin{proof} Let $s\in\Z$ and $f$ be a Lipschitz function with compact support. Let $L$ be a ball with $2^s\leq r_L\leq \gamma_1 2^s$ and $b$ be a $(1,r)$ atom supported in some ball $B=B(c_B,r)$ of radius $r\leq 2^s$. In order to verify a. of Definition~\ref{d.LPimprove} we write for $x\in \gamma_2 L$
\[
|T(s)[f\ind_L](x)|\leq C_T \int_{2^s\leq \d(x,y)<2^{s+1}} \frac{|f(y)\ind_L(y)|}{V(x,y)} \,\d y .
\]
Now $V(x,y)=|B(x,\d(x,y))|\geq |B(x,2^s)|\gtrsim |L|$  whenever $x\in\gamma_2 L$ by the doubling assumption and the observation that $L\subset B(x,c2^s)$ for some suitable constant $c>0$. 	This proves
\[
\|T(s)[f\ind_L]\|_{L^\infty(\gamma_2 L)} \lesssim_{T,\mathbb X} \l f \r_{1,L}
\]
as desired in order to verify a. of the definition. In order to verify b. we use the localization property \eqref{eq:localization} and the cancellation of $b$ to estimate
\[
\begin{split}
& \left|  \int[T(s)[f\ind_L](x)b(x) \,\d x\right|\leq \int_L  \int_{c_oL} | K_s(x,y)-K_s(c_B,y) ]f(y)b (x)| \, \d x\,  \d y
\\
&\qquad \lesssim  \omega\left(\frac{r}{2^s}\right) \int_{L} \int_{c_oL}\frac{1}{V(x,y)}|f(y)|  |b|\ind_{\{ 2^s\leq \d(x,y)<2^{s+1}\}}(x,y) \, \d x\, \d y
\\
& \qquad \lesssim  \omega\left(\frac{r}{2^s}\right) |L| \langle f\rangle_{1,L} \l b\r_{1,c_o L}
\end{split}
\]
by noting again that we can replace the term $V(x,y)$ by $|L|$ for 	$x\in L$ and $2^s<\d(x,y)\leq  2^s$. This proves the $(1,\infty)$-improving property according to Definition~\ref{d.LPimprove} with the same modulus of continuity as in the definition of $\hspace{.1em}T$; in particular here $\omega$ is assumed to satisfy the Dini condition. The conclusion for $T^*$ follows since $T$ is essentially self adjoint.
\end{proof}
%%%%%%%%%%%%%%%%%%%%%%%%%%%%%% PROOF PROOF PROOF

Combining Lemma~\ref{lem:imprczo} with Theorem~\ref{th:SD} immediately yields sparse domination theorem for Dini-Calder\'on-Zygmund operators. Of course this result is known, see for example \cite{Kara,VZK}. However, our proof bypasses the usage of dyadic systems in spaces of homogeneous type,  unlike previous approaches.

%%%%%%%%%%%%%%%%%%%%%%%%%%%%%% THEOREM THEOREM THEOREM
\begin{theorem} \label{thm:czo} Let $({\mathbb{X}},\d,|\cdot|)$ be a space of homogeneous type and $T$ be a  Calder\'on--Zygmund operator on $({\mathbb{X}},\d,|\cdot|)$ which is bounded on some $L^p(\mathbb X)$, $1<p<\infty$, with Dini modulus of continuity. Then, for all $f_1,f_2\in \mathrm{Lip}({\mathbb{X}})$ with compact support there exists a sparse collection $\mathcal B$ consisting of $\d$-balls such that
\[
|\l T f_1, f_2 \r| \lesssim (C_p+\|\omega\|_{\mathrm{Dini}})\sum_{B \in \mathcal B} |B| \l f_1 \r_{1,B}  \l f_2 \r_{1, B}.
\]
The implicit constant depends on the homogeneous metric structure of $({\mathbb{X}},\d,|\cdot|)$ and on the constants in the kernel assumptions for $T$.
\end{theorem}
%%%%%%%%%%%%%%%%%%%%%%%%%%%%%% THEOREM THEOREM THEOREM

%%%%%%%%%%%%%%%%%%%%%%%%%%%%%% SECTION SECTION SECTION
\subsubsection{Geometric maximal operators} A somewhat trivial application of Theorem~\ref{th:max} provides a sparse domination theorem for geometric maximal operators in metric spaces. For this, consider the maximal operator
\[
\mathrm{M}f(x)\coloneqq \sup_{x\in B} \intav_B |f(y)|\, \d y,\qquad x\in\mathbb X.
\]
First of all, note that if $B=B(c_B,r_B)$ with $2^s <r_B\leq 2^{s+1}$ then the doubling property of the measure $|\cdot|$ implies that
\[
\intav_ B |f(y)|\, \d y \lesssim \intav _{B(c_B,2^{s+1})} |f(y)|\,\d y
\]
and so we can assume that all the balls in the definition of $\mathrm{M}$ have dyadic radii. Now,  define the single scale average
\[
T(s)f(x)\coloneqq \sup_{\substack{B\ni x\\ r_B=2^s}} \intav_B |f|,\qquad \mathrm{M}f \lesssim T_\star f\coloneqq \sup_{s\in\Z} |T(s)f|,\qquad x\in\mathbb X.
\]
A well known procedure allows us to approximate $T(s)$ by a smoother operator. Take a function $f$ which is bounded and compactly supported in some ball $B$ with $r_B=2^{s_B}$ and fix some scale $ s\in \Z$ with $s\leq s_B$. Since we are working on a homogeneous space we can cover $B$ by a union of balls $L_\tau\coloneqq B(c_\tau,2^s)$ so that for every $\rho>1$ we have $\sum_{\tau }\ind_{\rho L_\tau}\lesssim 1$; see Lemma~\ref{l.cover}. Then one easily constructs a $\sim 1/2^s$-Lipschitz partition of unity $\{\psi_\tau\}_{\tau}$, $0\leq \psi_\tau \leq 1$, subordinate to the cover $\{L_\tau\}_\tau$, so that $\psi_\tau\gtrsim 1$ on every ball $c_1L_\tau$ and $\mathrm{supp}\psi_\tau \subset c_2 L_\tau$ for each $\tau$ and some structural constants $c_2>c_1>1$. The single scale operator $T(s)$ can be approximated in the form
\[
T(s)f(x)\lesssim \sum_\tau \psi_\tau (x) \intav_{B(c_\tau, c_1 2^s)} |f(y)|\, \d y\lesssim \sum_\tau \frac{\psi_{\tau }(x) }{|B(c_\tau,c_1 2^s)|} \int |f(y)| \psi_{\tau}(y) \, \d y\eqqcolon A(s)|f|(x).
\]
The process above is a version of discrete convolution which is a standard tool in harmonic analysis on homogeneous spaces; see for example \cite{KiAl}. 

%%%%%%%%%%%%%%%%%%%%%%%%%%%%%% LEMMA LEMMA LEMMA
\begin{lemma} 
	The operator $A(s)$ defined above is $(1,\infty)$-improving for all $s\in \Z$.
\end{lemma}
%%%%%%%%%%%%%%%%%%%%%%%%%%%%%% LEMMA LEMMA LEMMA

%%%%%%%%%%%%%%%%%%%%%%%%%%%%%% PROOF PROOF PROOF
\begin{proof} Note that for every ball $ L$ of radius $r_L\simeq 2^s$ we have
\[
\sup_{x\in \gamma_2 L}|A(s)[f\ind_L](x)| \lesssim\sum_{\tau:\, c_2 L_\tau \cap \gamma_2 L \neq \varnothing}\frac{1}{|B(c_\tau,c_1 2^s)|}\int_{c_2L_\tau \cap L} |f(y)|\,\d y \lesssim \l f\r_{1,L}
\]
since all balls in the sum have comparable radius, they intersect $L$, and they have finite overlap. This proves a. of Definition~\ref{d.LPimprove}.

In order to prove b. of Definition~\ref{d.LPimprove} we consider for each $s$ the duality form
\[
\l A(s)[ f\ind_L],  b\r =\sum_\tau \frac{1}{|B(c_\tau,c_1 2^s)|}\int \int  \psi_\tau(x)\psi_\tau (y) f(y)\ind_L(y) b(x)\, \d y\, \d x
\]
where $r_L\eqsim 2^s$ and $b$ is a $(1,r)$ atom supported in $B=B(c_B,r)$ with $r\leq 2^s$. Then one easily verifies the $(1,\infty)$-improving property of Definition~\ref{d.LPimprove}. Indeed we have
\[
\begin{split}
& |\l A(s) [f\ind_L], b\r| \lesssim \sum_{\tau}  \frac{1}{|B(c_\tau,c_1 2^s)|}\int \bigg(\int  \psi_\tau(y)f(y)\ind_L(y) \d y \bigg) \psi_\tau (x) b(x) \, \d x\, 
\\
& \qquad \lesssim  \sum_{\tau:\,  {\substack{c_2L_\tau \cap B\neq \varnothing\\ c_2 L_\tau \cap L\neq \varnothing}}}  \frac{1}{|B(c_\tau,c_1 2^s)|}\int_B \bigg(\int_{L\cap c_2 L_\tau}  |\psi_\tau(y)f(y)|  \d y \bigg) |\psi_\tau (x)-\psi_\tau(c_B)|| b(x)|\, \d x.
\end{split}
\]
Observe that $B\cap c_2 L_\tau\neq \varnothing$ for all $\tau$ in the sum that yield non-zero terms and $r_B\leq 2^2\eqsim r_L$ so that $B\subset L_\tau$ for some $\tau$ in the sum. At the same time all $\tau$ in the sum that yield non-zero term must also satisfy $c_2L_\tau \cap L\neq \varnothing$ and these balls have comparable radii, so that $\cup_{\tau }L_\tau \subset \gamma_2 L$ for some structural constant $\gamma_2$. Combining these facts with the Lipschitz condition on the functions $\psi_\tau$ and the finite overlap of the balls $\{\rho L_\tau\}$ readily implies that
\[
 |\l A(s)  [f\ind_L], b\r|\lesssim \\ \frac{r}{2^s} |L|\l f \r_{1, L} \l b\r_{1,\gamma_2 L}
\]
which is b. of Definition~\ref{d.LPimprove}.
\end{proof}
%%%%%%%%%%%%%%%%%%%%%%%%%%%%%% PROOF PROOF PROOF

Combining the lemma above with Theorem~\ref{th:max} and the fact that $\mathrm{M}f\lesssim \sup_s |A_s(|f|)|$ immediately yields the following sparse domination result.

%%%%%%%%%%%%%%%%%%%%%%%%%%%%%% COROLLARY COROLLARY COROLLARY
\begin{corollary} \label{c:geomax}For every $f_1,f_2$ bounded with compact support there exists a sparse collection $\mathcal B$ such that
	\[
	|\l \mathrm M f_1,f_2\r |\lesssim  \sum_{B\in\mathcal B}|B| \l f_1 \r_{1,B} \l f_2 \r_{1, B}.
	\]
\end{corollary}
%%%%%%%%%%%%%%%%%%%%%%%%%%%%%% COROLLARY COROLLARY COROLLARY

%%%%%%%%%%%%%%%%%%%%%%%%%%%%%% SECTION SECTION SECTION
\subsection{Singular Radon transforms along polynomial manifolds} As anticipated in the introduction, our focus is on Radon transforms as examples of $L^p$-improving operators, in particular singular integrals along free monomial varieties. For this reason we focus on the next paragraph on metric spaces of the form $(\R^n,\d,|\cdot|)$, where $|\cdot|$ is the Lebesgue measure.

%%%%%%%%%%%%%%%%%%%%%%%%%%%%%% SECTION SECTION SECTION
\subsubsection{Homogeneous norms on $\R^n$} Consider the metric space $(\R^n,\d,|\cdot|)$ where $|\cdot|$ denotes the Lebesgue measure. In particular the underlying space is a vector space and we have translations. Furthermore we will assume that the metric $\d$ is given by a quasi-norm $\rho:\R^n\to [0,\infty)$ and that there exists a dilation structure $\delta_t:\R^n\to \R^n$, $t>0$, with respect to which the quasi-norm $\rho$ is homogeneous
\[
\d(x,y)\coloneqq \rho(x-y),\qquad \rho(\delta_tx)=t \rho(x),\qquad x,y\in\R^n,\quad t>0.
\]
For the purposes of this paragraph it will be enough to consider the special case that there exist $\alpha_1,\ldots,\alpha_n>0$ such that
\[
\delta_t(x_1,\ldots,x_n)\coloneqq (t^{\alpha_1}x_1,\ldots ,t^{\alpha_n}x_n),\qquad (x_1,\ldots,x_n)\in\R^n,\quad t>0.
\]
One of many equivalent quasi-norms compatible with $\delta_t$ can be defined as
\begin{equation}\label{eq:rho}
\rho(x)\coloneqq \Big(\sum_{j=1} ^n |x_j|^{\frac{2}{\alpha_j}}\Big)^{\frac12},\qquad x=(x_1,\ldots,x_n)\in\R^n,
\end{equation}
and  $\rho$ is homogeneous with respect to $\delta_t$. Clearly $\rho$ is symmetric and satisfies a quasi-triangle inequality. Furthermore for any ball $B(x,r)$ given by $\rho$ we have $|B(x,r)|=|B(0,r)|\eqsim_n r^\alpha$ where $\alpha\coloneqq \alpha_1+\cdots+\alpha_n$ will be referred to as the \emph{homogeneous dimension} of $(\R^n,\rho,|\cdot|)$. In this context will will write $(\R^n,\rho,|\cdot|)$ for the homogeneous metric structure on $\R^n$ described by these definitions. We note that this setup is classical and further details can be found in several references, see for example \cites{SW,stein,Par} and the references therein.

This scenario is particularly useful for the applications to singular operators given by integration against a measure supported on appropriate sub-manifolds of $\R^n$. For this reason we shall show an alternative way to deduce the $(p_1,p_2 ')$-improving property, which is arguably the most crucial assumption in Theorem~\ref{th:SD} given above.

As mentioned in the introduction, several operators of interest, such as singular integrals given by convolution with measures supported on lower dimensional sets, satisfy a stronger $(p_1,p_2 ')$-improving property described in Lemma~\ref{lem:thenewlemma} below. Our first task here is to deduce the $(p_1,p_2 ')$-improving property of Definition~\ref{d.LPimprove} in that case.

%%%%%%%%%%%%%%%%%%%%%%%%%%%%%% LEMMA LEMMA LEMMA
\begin{lemma}\label{lem:thenewlemma} Consider the space $(\R^n,\rho,|\cdot|)$ and a dilation semi-group $\{\delta_{r}\}_{r>0}$ such that the quasi-norm $\rho$ is homogeneous with respect to $\delta_r$. For each $s\in\Z$ let $m^s$ be a Borel probability measure supported on the unit $\rho$-ball of $\,\R^n$ and define the Borel measures $\widehat{\mu_s}(\xi)\coloneqq \widehat{m^s}(\delta_{2^s} \xi)$ and $T(s)f\coloneqq f*\d\mu_s$. The following hold.
\begin{itemize}
	\item [(i)] Suppose that $|\widehat {m^0}(\xi)|\lesssim |\xi|^{-\beta}$ for some $\beta>0$ and that $T(0):L^{p_1}\to L^{p_2'}$ for every $(p_1 ^{-1},p_2 ^{-1})\in\Omega$ for some open set $\Omega$. Then $T\sim \sum_s T(s)$ is $(p_1,p_2')$-improving and $T^*\sim  \sum_s T(s) ^*$ is $(p_2,p_1 ')$-improving in the sense of Definition~\ref{d.LPimprove} for all indices in the same open set $\Omega$.
	\item [(ii)] If the conclusion of either Theorem~\ref{th:SD} or Theorem~\ref{th:max} holds for the truncations \[
	T_\sigma^\tau= \sum_{\sigma\leq s<\tau} T(s)\] or for the maximal operator \[ T_\star f=
	\sup_{s\in \mathbb Z} |T(s)f|\]  and every $(p_1 ^{-1},p_2 ^{-1})\in\Omega$ for some open set $\Omega$,  then $T(0):L^{p_1}\to L^{p_2 '}$  for all indices in the same open set $\Omega$.
\end{itemize}
\end{lemma}
%%%%%%%%%%%%%%%%%%%%%%%%%%%%%% LEMMA LEMMA LEMMA

%%%%%%%%%%%%%%%%%%%%%%%%%%%%%% PROOF PROOF PROOF
\begin{proof} We begin with (i). The first step is to note the following scale-invariant continuity estimate: for every scale $s\in \Z$ and $y\in B(0,c2^s)$ and every $\rho$-ball $L$ of radius $r_L\eqsim 2^s$  
\begin{equation}\label{eq:cont}
	\left\langle   [ T(s)- \mathrm{Tr}_y T(s)] (f\ind_L)   \right\rangle_{p_2',c_oL}  \leq \omega\left(\rho(\delta^{-1}_{2^s} y)\right)   \langle f \rangle_{p_1,L},\qquad (p_1 ^{-1},p_2 ^{-1})\in\Omega,
\end{equation}	
where   $[\mathrm{Tr}_yg](x)\coloneqq g(x-y)$ for $x,y\in\R^n$, $\omega(t)=t^\eps$ for $t$ small and $\eps$ depends on the assumptions for $T$, the quasi-metric structure implied by $\rho$ and the exponents $p_1,p_2$. Now by translation invariance and scale invariance it suffices to prove \eqref{eq:cont} for $s=0$ and the ball $L=B(0,1)$. This estimate for $T(0)$ without the decay term $\omega\left(\rho(\delta^{-1}_{2^s} y)\right) $ is then a direct consequence of the assumption in (i).  On the other hand we have that
\[
\|T(s)- \mathrm{Tr}_y T(s):L^2(\R^n)\to L^2(\R^n)\| \leq \sup_{\xi\in\R^n} |(e^{-iy\cdot \xi}-1)\widehat{m^0}(\xi)|\lesssim \rho(y)^{\eps'}
\]
for some $\eps'>0$ depending on $\beta$ in the decay assumption for $\widehat{m^0}$ and the constants involved in the definition of $\rho$; this last dependence comes from the fact that we implicitly used the estimate $|y|\leq \rho(y)^c$ for $|y|\lesssim 1$. Since $\Omega$ is open we can use interpolation to conclude \eqref{eq:cont} with a power modulus of continuity for the same open set of indices $\Omega$. 

In order to complete the proof of (i) we show that \eqref{eq:cont} implies Definition~\ref{d.LPimprove}. Note that a. of Definition~\ref{d.LPimprove} is an immediate consequence of the $L^{p_1}\to L^{p_2 '}$ hypothesis for $T(0)$ and scale invariance. In order to prove b. of Definition~\ref{d.LPimprove} we let $b$ be a $(p_2,r)$-atom where $b\in L^{p_2}(B)$, $B=B(c_B,r)$ is a ball, $\int_{B} b=0$, and $r \leq 2^s$. Let $L$ be a ball of radius $r_L\eqsim 2^s$. By the mean zero condition on  $b$ we have  
\[
\begin{split}
\l T(s)	(f\ind_L), b  \r& = \l T(s)(f\ind_L),  b \r
\\
& =   \frac{1}{|B|}\int_{B} \bigg(\int_{B} [T(s)(f\ind_L )(x)-\mathrm{Tr}_{x-x'}T(s)(f\ind_L )(x)] b(x)\, \d x' \bigg)\, \d x
\\
&=   \frac{1}{|B|}\int_{B} \bigg(\int_{x-B} [T(s)(f\ind_L )(x)-\mathrm{Tr}_{y}T(s)(f\ind_L )(x)] b(x)\, \d y \bigg)\, \d x.
\end{split}
\]
Now we remember that $|\cdot|$ is Lebesgue measure and each $\rho$-ball $B$ has radius $r_B$ so that $|B|=|B(0,r)|\eqsim r_B ^\alpha$ for every $j$. Furthermore, for $x\in B$ we have that $x-B\subset B-B\subset B(0,c_2r)$ for some constant $c_2$ depending on the quasi-metric constant. These remarks and the previous calculation show that
\[
\begin{split}
\l T(s)	(f\ind_L), b  \r &=\frac{1}{|B(0,r)|} \int \bigg( \int_{B(0,c_2 r)}[T(s)(f\ind_L )(x)-\mathrm{Tr}_{y}T(s)(f\ind_L )(x )]  b (x)\ind_{x-B}(y)    \, \d y \bigg) \, \d x
\\
& =\frac{1}{|B(0,r)|}  \int_{B(0,c_2r)}  \Big\l [T(s)-\mathrm{Tr}_{y}T(s)], (f\ind_L ) \ind_{c_3 L}    b \ind_{y+B}  \Big\r  \, \d y
\end{split}
\]
with $c_3$ depending on the quasi-metric constant. Note that in the last line we are allowed to insert the indicator $\ind_{c_3L}$. Indeed, since $T$ satisfies \eqref{eq:localization} and $r_L\eqsim 2^s$, $r\leq 2^s$, we have  for $y\in B(0,c_2r)\subset B(0,c_2 2^s)$ that
\[
\mathrm{supp}[ T(s)-\mathrm{Tr}_{y}T(s)]\subset c_o L \cup (c_o L+y)\subset c_3 L
\]
for some constant $c_3$ depending on $c_2,c_o$, the implicit constants in $r_L\eqsim 2^s$, and the quasi-metric constant of $\d$. Using \eqref{eq:cont} we can now conclude
\[
\begin{split}
|\l T(s)	(f\ind_L), b  \r|&\lesssim_{\mathbb X} \intav_{B(0,c_2r)} \omega(\rho(\delta^{-1} _{2^s}y))|L|^{\frac{1}{p_2 '}-\frac{1}{p_1}}\|f\ind_L\|_{p_1}\big\|  \ind_{c_3L} b \big\|_{p_2} \, \d y
\\
& \lesssim \omega(c_2r/2^s) |L| \l f \r_{p_1,L} \l b\r_{p_2,c_3 L}.
\end{split}
\]
This proves that $T$ is $(p_1,p_2 ')$-improving with modulus $\omega(t)=t^\eps$ for some $\eps>0$ and indices $(p_1 ^{-1},p_2 ^{-1})\in \Omega$.  

We now prove (ii). By Proposition~\ref{prop:partial} we can conclude the scale-invariant estimate
\[
\l T(0)(f\ind_L)\r_{p_2 ',L}\lesssim \l f \r_{p_1,L}
\]
whenever $L$ is a ball of radius $r_L\eqsim 1$. But then one can decompose $\R^n$ into a finitely overlapping collection of balls $\{L_\tau\}_{\tau}$ with $r_{L_\tau} \eqsim 1$ for all $\tau$ which yields
\[
\begin{split}
\| T(0)f\|_{L^{p_2 '}} ^{p_2 '} &\lesssim \sum_\tau \int_{c_o L_\tau} |T(0)[f\ind_{L_\tau}]|^{p_2 '} \lesssim \sum_{\tau}|L_\tau|^{1-\frac{p_2'}{p_1}} \l f\r_{p_1,L_\tau} ^{\frac{p_2'}{p_1}}\eqsim \sum_\tau \Big( \int_{L_\tau}|f|^{p_1}\Big)^{\frac{p_2'}{p_1}} \lesssim \|f\|_{p_1} ^{p_2 '}
\end{split}
\]
since $|L_\tau|\simeq 1$ and $p_2'\geq p_1$.
\end{proof}
%%%%%%%%%%%%%%%%%%%%%%%%%%%%%% PROOF PROOF PROOF

%%%%%%%%%%%%%%%%%%%%%%%%%%%%%% SECTION SECTION SECTION
\subsubsection{$L^p$-improving properties for kernels with Fourier decay} With Lemma~\ref{lem:thenewlemma} in hand we now derive as an application the $L^p$-improving property for a class of operators which are given by convolution with measures supported on lower dimensional manifolds on $\R^n$. The main assumption is the suitable Fourier decay of the measures, which ultimately relies on suitable curvature assumptions on their support.

%%%%%%%%%%%%%%%%%%%%%%%%%%%%%% LEMMA LEMMA LEMMA
\begin{lemma}\label{lem:Lpimp*mu} Consider the space $(\R^n,\rho,|\cdot|)$ where $\rho$ is the quasi-norm given in \eqref{eq:rho} and $|\cdot|$ denotes the Lebesgue measure. For each $s\in \Z$ let $m^s$ be a Borel measure supported on the metric ball $B(0,1)$. Assume that there exists $\beta>0$ such that uniformly in $s\in \Z$ we have
\begin{equation}
\label{e:decay}
 \left|\widehat{m^s}(\xi)\right|\lesssim |\xi|^{-\beta}, \quad \xi\in\R^n,\qquad \int_{\R^n}|\d m^s|\leq 1.
\end{equation}
For each $s\in\Z$ define the scaled measure $\mu_s$ as $\widehat {\mu_s}(\xi)\coloneqq \widehat{m^s}(2^s\xi)$ for $\xi \in \R^n$ and $T(s)f\coloneqq f*\d\mu_s$. Then there exists a modulus of continuity $\omega(t)=|t|^{\varepsilon}$ for some $\varepsilon>0$ depending only on $n,\beta$ and the metric $\rho$ such that $T\sim \sum_s T(s)$ is $(p_1,p_2')$-improving at scale $s$ and $T^*\sim\sum_s T^*(s)$ is $(p_2,p_1' )$-improving in the sense of Definition~\ref{d.LPimprove} whenever $1< p_1,p_2<2$ and $\frac{\beta+n}{n}>\frac{1}{p_1}+\frac{1}{p_2}$.
\end{lemma}
%%%%%%%%%%%%%%%%%%%%%%%%%%%%%% LEMMA LEMMA LEMMA

%%%%%%%%%%%%%%%%%%%%%%%%%%%%%% PROOF PROOF PROOF
\begin{proof} By Lemma~\ref{lem:thenewlemma} it suffices to prove that $\|T(0):L^{p_1}\to L^{p_2 '}\|\lesssim 1$ for $p_1,p_2$ as in the conclusion of the lemma. This in turn will follow by interpolation with the easy $L^2\to L^2$ bound and the estimate
\begin{equation}\label{eq:endpoint}
\|T(0) g\|_{L^{q_2 {'}}(\R^n)}\lesssim \|g\|_{L^{q_1}(\R^n)},\qquad 1<q_1\leq 2, \quad 2\leq q_2 '<\infty, \quad \frac{1}{q_1}+\frac{1}{q_2}=1+\frac\beta n.
\end{equation}
Assume first that $q_2 ' >2$. Using the Hardy-Littlewood-Sobolev inequality  we get the following whenever $\beta\geq \gamma >0$ and $\frac{1}{2}= \frac{1}{q'_{2}}+\frac{\gamma}{n}$
\[
\begin{split}
 \|[T(0)g]\|_{L^{q'_2}(\R^n)}& = \||\nabla|^{-\gamma}(|\nabla|^\gamma [T(0)g])\|_{L^{q'_2}(\R^n)} \eqsim_{s,n} \||x|^{\gamma-n}* (|\nabla|^\gamma [T(0)g])\|_{L^{q'_2}(\R^n)}
 \\
& \lesssim \||\nabla|^\gamma [T(0)g]\|_{L^{2}(\R^n)}.
\end{split}
\]
Now the right hand side in the display above can be further estimated as follows using Plancherel's theorem
\[
\begin{split}
\||\nabla|^\gamma [T(0)g]\|_{L^{2}(\R^n)} ^2 &=\int_{\R^n} |\hat g(\xi)|^2 |\xi|^{2\gamma} |\widehat{\mu_1}(\xi)|^{2}\,\d \xi	
\\
& \leq \int_{|\xi|\leq 1}|\hat g(\xi)|^2 \,\d \xi+\int_{|\xi|>1}  |\hat g(\xi)|^2 |\xi|^{2(\gamma-\beta)}\,\d \xi.
\end{split}
\]
The first summand above can be estimated by H\"older's inequality combined with the Hausdorff-Young inequality by $\|g\|_{L^{q_1}}$ for any $1\leq q_1\leq 2$. Furthermore, if $q_1=2$ then the claim follows by the estimate above for $\beta=\gamma$. Now suppose that $1<q_1< 2$ and choose $0<\gamma<\beta$ so that $\beta-\gamma = n(1/q_1-1/2)$ which is always possible. Then we can estimate
\[
\int_{|\xi|>1}  |\hat g(\xi)|^2 |\xi|^{2(\gamma -\beta)}\,\d \xi\lesssim \int_{\R^n}  |\hat g(\xi)|^2 (1+|\xi|^2)^{(\gamma-\beta)}\,\d \xi \lesssim \|g\|_{L^{q_1}(\R^n)} ^2
\]
by the dual Sobolev embedding theorem. Finally if $q_2 ' =2$ then the claim follows directly by the $L^2$-estimate above and the dual form of the Sobolev embedding theorem which embeds $H^{-\beta}$ into $L^{q_1}$ for $\beta=n(1/q_1-1/2)$. This proves \eqref{eq:endpoint} and thus the $(p_1,p_2 ')$-improving property for $T\sim\sum_s T(s)$. As $T$ is essentially self-adjoint we get for free that $T^*$ satisfies the dual of \eqref{eq:endpoint} and is thus $(p_2,p_1')$-improving and the proof is complete.
\end{proof}
%%%%%%%%%%%%%%%%%%%%%%%%%%%%%% PROOF PROOF PROOF

We recall the known fact that operators given at each \emph{dyadic scale} by a convolution with a measure that has Fourier decay as in Lemma~\ref{lem:Lpimp*mu} are bounded on $L^p(\R^n)$. 

%%%%%%%%%%%%%%%%%%%%%%%%%%%%%% LEMMA LEMMA LEMMA
\begin{lemma}\label{lem:Lpbound*mu}  Let $\{m^s\}_{s\in\Z}$ be a sequence of Borel measures all supported on the metric ball $B(0,1)$ and such that $\int_{\R^n}|\d m^s|\leq 1$. Assume that there exists $\beta>0$ such that for every $s\in \Z$
\[
 \left|\widehat{m^s}(\xi)\right|\lesssim |\xi|^{-\beta} , \qquad \xi\in\R^n.
\]
For each $s\in\Z$ define the scaled measure $\mu_s$ as $\widehat {\mu_s}(\xi)\coloneqq \widehat{m^s}(2^s\xi)$ for $\xi \in \R^n$.
\begin{itemize}
\item[(i)] The operator $T_\star f\coloneqq \sup_s |f*\d \mu_s|$ extends to a bounded operator on $L^p(\R^n)$ for all $1<p<\infty$.
\item[(ii)] If in addition we have that $\int_{\R^n}\d m^s=0 $ for all $s\in \Z$ then $Tf\coloneqq \sum_{s\in\Z}f*\d\mu_s$ extends to a bounded operator on $L^p(\R^n)$ for all $1<p<\infty$. The same holds uniformly for the truncations $T_{\sigma} ^\tau f =\sum_{\sigma\leq s <\tau}T(s)f$.
\end{itemize}	
\end{lemma}
%%%%%%%%%%%%%%%%%%%%%%%%%%%%%% LEMMA LEMMA LEMMA
We omit the well known proof of Lemma~\ref{lem:Lpbound*mu} and refer the reader to \cite[\S XI.2.5]{stein} for the maximal version of (i), and to \cite[\S XI.4.4]{stein} for the singular integral version of (ii) above. 

With the ingredients above it is now easy to conclude a sparse domination theorem for operators given by convolutions with suitable measures possessing Fourier decay as above. 

%%%%%%%%%%%%%%%%%%%%%%%%%%%%%% THEOREM THEOREM THEOREM
\begin{theorem}\label{thm:sparseFT} Consider the space $(\R^n,\rho,|\cdot|)$ where $\rho$ is a quasi-norm and $|\cdot|$ denotes the Lebesgue measure. For each $s\in \Z$ let $m^s$ be a Borel measure supported on the metric ball $B(0,1)$. Assume that there exists $\beta>0$ such that 
\[
\sup_{s\in \mathbb Z}\int_{\R^n}|\d m^s|\leq 1, \qquad
\sup_{s\in \mathbb Z} \sup_{\xi \in \R^n} |\xi|^\beta\left|\widehat{m^s}(\xi)\right|\leq  1
\]
 For each $s\in\Z$ define the scaled measure $\mu_s$ as $\widehat {\mu_s}(\xi)\coloneqq \widehat{m^s}(\delta_{2^s}\xi)$ and let 
\[
T_{\star,\sigma,\tau} f \coloneqq \sup_{\sigma<s\leq \tau} |f*\d\mu_s|,\qquad T_\sigma ^\tau f = \sum_{\sigma\leq s<\tau } f* d\mu_{s}.
\]
\begin{itemize}
\item[(i)] For all $f_1,f_2\in\mathcal S(\R^n)$ with compact support and $\sigma,\tau\in\Z$ with $\sigma<\tau$ there exists a sparse collection $\mathcal B_{\star,\sigma,\tau}$ consisting of balls $B$ with $2^\sigma \leq r_B \leq \tau$ such that
\[
|\l T_{\star,\sigma,\tau} f_1,f_2 \r|\lesssim \sum_{B\in\mathcal B_{\star,\sigma,\tau}}|B|\l f_1\r_{p_1,B}\l f_2 \r_{p_2B}
\]
whenever $\frac{\beta+n}{n}>\frac{1}{p_1}+\frac{1}{p_2}\geq 1$.
\item[(ii)] If in addition we have that \[\int\d m^s=0\qquad \forall s\in\Z,\] then  for all $\sigma,\tau\in\Z$ with $\sigma<\tau$ and for all $f_1,f_2\in\mathcal S(\R^n)$ with compact support there exists a sparse collection $\mathcal B_{\sigma,\tau}$  consisting of balls $B$ with $2^\sigma\leq r_B \leq 2^\tau$ such that
\[
|\l T_{\sigma} ^\tau f_1,f_2 \r|\lesssim \sum_{B\in\mathcal B}|B|\l f_1\r_{p_1,B}\l f_2 \r_{p_2B}
\]
whenever $1\leq p_1,p_2\leq 2$ and $\frac{\beta+n}{n}>\frac{1}{p_1}+\frac{1}{p_2}\geq 1$. 
\end{itemize}
The corresponding conclusions hold for the untruncated versions with sparse collections consisting of balls of all radii.
\end{theorem}
%%%%%%%%%%%%%%%%%%%%%%%%%%%%%% THEOREM THEOREM THEOREM

%%%%%%%%%%%%%%%%%%%%%%%%%%%%%% SECTION SECTION SECTION
\subsubsection{Sparse domination for singular Radon transforms} We culminate the considerations of this section by describing a class of singular Radon transforms given by convolution with measures supported on polynomial subvarieties of $\R^n$. To make this specific we fix some positive integer $d$ and consider the polynomial map
\begin{equation}\label{eq:rhogamma}
\gamma:\R^k \to \R^N,\qquad \gamma(t)=(t^\alpha)_{1\leq |\alpha|\leq d},
\end{equation}
where $N$ is the number of monomials $t^\alpha=t^{\alpha_1}\cdots t^{\alpha_k}$ with $|\alpha|=\alpha_1+\cdots+\alpha_k\leq d$. It is convenient to describe points $x\in \R^N$ in the form $x=(x_\alpha)_{1\leq |\alpha|\leq d}$.  With these conventions in hand we define dilations $\delta_r((x_\alpha)_\alpha)\coloneqq ((r^{|\alpha|}x_\alpha)_\alpha)$ and $\rho$ by 
\[
\rho(x) =\bigg( \sum_{1\leq |\alpha| \leq d} |x_\alpha| ^\frac{2}{|\alpha|}\bigg)^\frac12,\qquad x=(x_\alpha)_\alpha \in \R^N,
\]
which is just formula \eqref{eq:rho} in current notation. We can always compare the quasi-norm $\rho$ with the Euclidean one by means of
\begin{equation}\label{eq:comparison}
\begin{cases}
 |x|^{\frac1d}\lesssim \rho(x) \lesssim   |x|,& \quad\text{if}\quad |x|> 1,\vspace{.5em}
\\
 |x| \lesssim \rho(x) \lesssim |x|^{\frac1d},&\quad \text{if}\quad |x|\leq 1.
\end{cases}
\end{equation}
The homogeneous dimension of $(\R^N,\rho,|\cdot|)$ is
\[
\Delta\coloneqq \sum_{1\leq |\alpha|\leq d} |\alpha|.
\]
Note the following basic behavior of $\rho(\gamma(t))$ with respect to dilations: for every $r>0$ we have
\begin{equation}\label{eq:dil}
r\rho(\gamma(t))= \rho(\delta_r \gamma(t))=\rho(\gamma(rt)),\qquad rt=(rt_1,\ldots,rt_k)\in\R^k.
\end{equation}
Now let $\Omega:\mathbb S^{k-1}\to \R$ be a $0$-homogeneous function with mean zero on $\mathbb S^{k-1}$ and  $\Omega \in C^{\infty}(\mathbb S^{k-1})$.  We define the singular Radon transform
\begin{equation}
\label{e:meo}
T_{\gamma}f(x)\coloneqq \mathrm{p.v.} \int_{\R^k} f(x-\gamma(t))\frac{\Omega(t)}{|t|^k}\d t,\qquad x\in\R^N,
\end{equation}
and we have that
\[
	T_{\gamma}f(x)=\sum_{s\in\Z}[T(s)f](x) \coloneqq  \sum_{s\in \Z}\int_{\R^k} f(x-\gamma(t))\psi\left(\frac{|t|}{2^s}\right)\frac{\Omega(t)\d t}{|t|^k} ,\qquad x\in\R^N,
\]
with $\psi\in\mathcal S(\R)$, $0\leq \psi\leq 1$, $\psi$ is compactly supported in $[1/2,4]$ and identically one in $[1,2]$ and such that $\sum_{s\in\Z} \psi\big(\frac{|t|}{2^s}\big)\eqsim 1$.
Consider the Borel measure $m$ defined as 
\[
\int_{\R^N} \phi(y) \,\d m(y) = \int_{\R^k} \phi(\gamma(t))\, \frac{\Omega_1(t)}{|t|^k} \d t\qquad \forall \phi \in\mathcal S(\R^n),\qquad \Omega_1(t) \coloneqq \Omega(t) \psi\big(|t|\big),
\]
which by \eqref{eq:comparison} is compactly supported in some $\rho$-ball in $\R^N$ of fixed radius and centered at $0$. By \eqref{eq:dil} we have that 
\[
T_\gamma f = \sum_{s\in \Z} T(s) f,\qquad T(s)f= f* \d\mu_s,\qquad \widehat{\d \mu_s}(\xi)\coloneqq \widehat m(\delta_{2^s} \xi),\qquad \xi \in \R^N.
\]
For each $s\in \Z$ the measure $\d\mu_s$ is a rescaling of the measure $\d\mu_0=m$ 
\[
\int_{\R^N}\phi(x)\, \d\mu_s(x)=\int_{\R^k} \phi(\delta_{2^s}\gamma(t))\frac{\Omega_1(t)}{|t|^k}\,\d t=\int_{\R^k} \phi(\gamma(t))\frac{\Omega(t)}{|t|^k}\psi\big(\frac{|t|}{2^s}\big)\,\d t\qquad \forall \phi \in\mathcal S(\R^N).
\]
With these definitions in mind and using \eqref{eq:dil} and \eqref{eq:comparison} it is routine to verify that $T$ obeys \eqref{eq:localization}. 
We also record the basic calculation
\[
\widehat{\d \mu_s}(\xi) = \int_{\R^k} e^{-i\xi \cdot \gamma(2^st)}\frac{\Omega(t)}{|t|^k}\psi(|t|)\,\d t =  \int_{\mathbb S^{k-1}} \bigg(\int_0 ^\infty e^{-i\xi \cdot \gamma(2^sr t')}\frac{\psi(r)\d r}{r} \bigg)\, \Omega(t')\d\sigma_{k-1} ,\quad \xi \in \R^N,
\]
whence 
\[
\int_{\R^N}\d\mu_s=\int_{\R^N} \d m = \widehat m(0) \eqsim \int_{\mathbb S^{k-1}} \Omega(t')\,\d\sigma_{k-1}(t')=0
\]
by our cancellation assumption for $\Omega$.  The previous calculation also implies that for all $s\in\Z$ we have $\|m\|=\|\d\mu_s \|\lesssim \|\Omega\|_{L^1(\mathbb S^{k-1})}$. 
The companion maximal operator is given as
\[
\mathcal M f(x)\coloneqq \sup_{r>0} \frac{1}{r^k} \int_{|t|\leq r} |f(x-\gamma(t))|\, \d t.
\]
As before letting $\psi_{2^s}(|t|)\coloneqq 2^{-ks}\psi(|t|/2^s)$ we can bound
\[
\mathcal M_{\gamma} f(x) \lesssim \sup_{s\in \Z} \frac{1}{2^{sk}} \int_{2^{s-1}\leq|t|<2^s}|f(x-\gamma(t)|\,\d t \lesssim \int|f(x-\gamma(t)|\psi_{2^s}(|t|)\,\d t   \eqqcolon \sup_s|f|*\d\nu_s
\]
with
\[
\int_{\R^n}\phi(t)\d\nu_s = \frac{1}{2^{sk}}\int_{\R^k} \phi(\gamma(t))\psi_{2^s}(|t|)\,\d t\qquad \forall \phi \in\mathcal S(\R^n).
\]
As before it can be easily seen that $\d \nu_s$ is the $\delta_{2^s}$ scaling of the measure $\nu=\nu_0$  so that $\widehat{\d\nu_s}(\xi)=\widehat{\d\nu}(\delta_{2^s} \xi)$ and $\d\nu$ is a compactly supported Borel measure with $\|\d\nu\|=\|\d\nu_s\|\lesssim 1$ for all $s\in\Z$.

We only miss one main ingredient in order to apply Theorem~\ref{thm:sparseFT} for the singular and maximal Radon transforms, $T$ and $\mathcal M$, respectively; that is,  the Fourier decay of the generating measure. However such estimates are standard in the context above since the polynomial map $\gamma$ is of finite type $d$.

%%%%%%%%%%%%%%%%%%%%%%%%%%%%%% LEMMA LEMMA LEMMA
\begin{lemma}\label{lem:fourierdecay} Let $\d \mu,\d\nu$ denote the Borel measures defined above. Then
	\[
	|\widehat{\d \nu}(\xi)|\lesssim |\xi|^{-\frac{1}{d}},\qquad |\widehat{\d\mu}(\xi)|\lesssim  |\xi|^{-\frac1d},\qquad \xi \in\R^N.
	\]	
\end{lemma}
%%%%%%%%%%%%%%%%%%%%%%%%%%%%%% LEMMA LEMMA LEMMA
The proof of the lemma above is classical and relies on the fact that the smooth polynomial map $\R^k\ni t\mapsto \gamma(t)$ is of finite type (at most $d$) at each point; see \cite[\S XI.2.2]{stein} and \cite[\S VII.3.2]{stein}.

Combining the estimates of Lemma~\ref{lem:fourierdecay} and the properties of the operators $T_\gamma,\mathcal M_{\gamma}$ yields the following sparse domination result.

%%%%%%%%%%%%%%%%%%%%%%%%%%%%%% COROLLARY COROLLARY COROLLARY
\begin{corollary} \label{c:radon1} Let $\gamma:\R^k\to \R^N$ be the map $\gamma(t)=(t^\alpha)_{1\leq |\alpha|\leq d}$ with $N$ denoting the dimension of the space spanned by the monomials of degree at most $d$, and let $\rho$ be given by \eqref{eq:rhogamma}. For every $f_1,f_2\in \mathcal S(\R^n)$ with compact support there exists sparse collections $\mathcal B_\star,\mathcal B$ such that
	\[
	|\l \mathcal M_\gamma f_1,f_2\r| \lesssim \sum_{B\in\mathcal B_\star}|B| \l f \r_{p_1,B}\l f_2\r_{p_2,B},\qquad 	|\l T_\gamma f_1,f_2\r| \lesssim \sum_{B\in\mathcal B}|B| \l f \r_{p_1,B}\l f_2\r_{p_2,B},
	\]
whenever $1\leq p_1,p_2\leq 2$ and $1\leq \frac{1}{p_1}+\frac{1}{p_2}\leq 1+\frac{1}{Nd}$. Corresponding statements hold for the truncated versions as above.
\end{corollary}
%%%%%%%%%%%%%%%%%%%%%%%%%%%%%% COROLLARY COROLLARY COROLLARY

Variations of the sparse domination result are possible with weaker conditions on $\Omega$ for example but we do not pursue those here. Furthermore one can provide a sparse domination theorem whenever some $L^{p_1}\to L^{p_2 '}$ improving property is known. We give one such example below.

%%%%%%%%%%%%%%%%%%%%%%%%%%%%%% COROLLARY COROLLARY COROLLARY
\begin{corollary} Let $\gamma:\R^2\to \R^5$ denote the polynomial map $\gamma(t)=(t_1,t_2,t_1 ^2,2t_1t_2,t_2 ^2)$ and define the singular Radon transform
	\[
		T_\sigma ^\tau f (x) \coloneqq\int_{2^\sigma\leq |t|< 2^\tau} f(x-\gamma(t)) \frac{\Omega(t)}{|t|^2}\, \d t,
	\]
with $\int_{\mathbb S^1}\Omega(t)\, \d \sigma_1(t)=0$ and $\Omega\in L^\infty(\mathbb S^1)$.
For every $f_1,f_2$ with compact support and every $\sigma,\tau\in\Z$ with $\sigma,\tau$ there exists a sparse collection $\mathcal B_{\sigma,\tau}$ consisting of $\rho$-balls  $B$ with $2^\sigma\leq r_B\leq 2^\tau$ such that
\[
\l T_\sigma ^\tau f_1,f_2\r \lesssim \sum_{B\in \mathcal B_{\sigma,\tau}} |B|\l f_1\r_{B,p_1}\l f_2 \r_{B,p_2}
\]
whenever $(\frac{1}{p_1},\frac{1}{p_2})$ is in the interior of the triangle with vertices $(0,1)$, $(1,0)$, and $(\frac58,\frac58)$. A similar estimate holds for the maximal operator
\[
\mathcal Mf(x)\coloneqq \sup_{r>0}\frac{1}{r^2}\int_{|t|\leq r} |f(x-\gamma(t))|\d t.
\]
 Furthermore no such sparse bound can hold outside the closed triangle with vertices as above.
\end{corollary}
%%%%%%%%%%%%%%%%%%%%%%%%%%%%%% COROLLARY COROLLARY COROLLARY

%%%%%%%%%%%%%%%%%%%%%%%%%%%%%% PROOF PROOF PROOF
\begin{proof} Let $T(s)$ denote the single scale operator $T_{s} ^{s+1}$. Observe that
	\[
	|T(s)f|\lesssim \|\Omega\|_{L^\infty(\mathbb S^1)} \frac{1}{2^{2s}}\int_{2^s \leq |t|<2^{s+1}} |f(x-\gamma(t))| \, \d t
	\]	
By \cite{gress} we know that $T(0)$ maps $L^{p_1}\to L^{p_2 {'}}$ whenever $(\frac{1}{p_1},\frac{1}{p_2})$ is in the open triangle of the statement. It is also well known that the measure $\d m$
\[
\int_{\R^5}\phi(x)\d m = \int_{\frac12\leq |t|<1} \phi(\gamma (t)) \frac{\d t}{t}
\]
satisfies $|\widehat m(\xi)|\lesssim |\xi|^{-\frac{1}{2}}$ as in Lemma~\ref{lem:fourierdecay}. Since $T(s)f=f*\d\mu_s$ with $\widehat{\d\mu_s}(\xi)\coloneqq \widehat m(\delta_{2^s}\xi)$, Lemma~\ref{lem:thenewlemma} shows that $T\sim\sum_s T(s)$ is $(p_1,p_2 {'})$ improving in the sense of Definition~\ref{d.LPimprove}. Furthermore these operators are singular Radon transforms along polynomial varieties of finite type so they are known to be bounded on $L^p(\R^5)$ for $1<p<\infty$; see for example \cite[\S XI]{stein}. The sparse domination follows by an application of Theorem~\ref{th:SD}.

We prove the sharpness of the sparse region for $T_\sigma ^\tau$ by recalling a well known example. Let $f_\delta\coloneqq \ind_{B(c_o,\delta)}$ with $\delta$ small and $c_o\coloneqq \gamma(\frac34)$. Consider also $\Omega$ bounded and with mean zero on $\mathbb S^1$ and such that $\Omega \equiv 1$ on the positive quadrant of $\R^2$. We can then easily calculate that $|[T(s)f_\delta](x)|= [T(s)f_\delta](x) \gtrsim \delta^2$ on the set of $x$ in the positive quadrant of $\R^5$ such that $|x-\gamma(t)|\leq \delta/2$ for some $t\in(1/2,1)$. The set of such $x$ has measure $\gtrsim \delta^{3}$ and so $T(0):L^{p_1}\to L^{p_2 '}$ implies that
\[
\delta^2 \delta^{\frac{3}{p_2'}}\lesssim \|T(0)f_\delta \|_{ L^{p_2 '}}\lesssim \|f_\delta\|_{L^{p_1}} \delta^\frac{5}{p_1}
\]
which together with the symmetric estimate which follows by self-duality yields the restrictions
\[
2+\frac{3}{p_2 '} \geq \frac{5}{p_1},\qquad 2+\frac{3}{p' _1}\geq \frac{5}{p_2}.
\]
The restrictions above describe the closure of the triangle in the statement so for $(p_1 ^{-1},p_2 ^{-1})$ outside the closed triangle the sparse domination result of the corollary has to fail. The example proving the sharpness of the sparse form for $\mathcal M$ is similar but simpler.
\end{proof}
%%%%%%%%%%%%%%%%%%%%%%%%%%%%%% PROOF PROOF PROOF

%%%%%%%%%%%%%%%%%%%%%%%%%%%%%% SECTION SECTION SECTION
\section{Whitney covers in geometrically doubling metric spaces and sparse collections}\label{sec2} In this section we describe the covering argument that will be employed in the proof of the main theorem as a way to obtain appropriate stopping balls. The covering argument is a \emph{Whitney-type decomposition} in suitable metric spaces. The construction of such Whitney coverings and allied issues occupies the first half of this section. In the second half we will use the Whitney decomposition in order to construct \emph{sparse collections} that will be central in the proofs of the main results of this paper.

%%%%%%%%%%%%%%%%%%%%%%%%%%%%%% SECTION SECTION SECTION
\subsection{Whitney covers in metric spaces}\label{sec:whit}  We recall the notion of a \emph{geometrically doubling metric space}. 

%%%%%%%%%%%%%%%%%%%%%%%%%%%%%% DEFINITION DEFINITION DEFINITION
\begin{definition}\label{def:geomdoub} We will say that the quasi-metric space $({\mathbb{X}},\d)$ is \emph{geometrically doubling} if there exists some positive integer $N$ such that every ball of radius $r$ may be covered by at most $N$ balls of radius $r/2$.
\end{definition}
%%%%%%%%%%%%%%%%%%%%%%%%%%%%%% DEFINITION DEFINITION DEFINITION

Before proceeding to the description of the Whitney covering lemma, some remarks are in order. Firstly we note that if $({\mathbb{X}},\d,|\cdot|)$ is a doubling quasi-metric measure space, then ${\mathbb{X}}$ is automatically geometrically doubling. Secondly, we note that the definition of the geometric doubling property does not really depend on the quasi-metric. Indeed, if $({\mathbb{X}},\d)$ has the geometric doubling property then so does $({\mathbb{X}},\d ')$ for any quasi-metric $\d '$ which is equivalent to $\d$. In that case the number $N$ appearing in the definition of geometric doubling will depend on the choice of quasi-metric; see \cite[\S 2.1]{AlvMi} for an extensive discussion on the geometry of quasi-metric spaces. We shall not pursue these subtle issues in the current paper as for us the consideration of a single quasi-metric in ${\mathbb{X}}$ will be sufficient, and the precise value of the relevant constant is unimportant.

We state below the Whitney-type covering result that will be used throughout the paper. In the formulation below this Whitney decomposition is contained in \cite[Theorem 2.4]{AlvMi}.

%%%%%%%%%%%%%%%%%%%%%%%%%%%%%% LEMMA LEMMA LEMMA
\begin{lemma}[Whitney-type decomposition, \cite{AlvMi}]\label{lem:whitney} Let $({\mathbb{X}},\d)$ be a geometrically doubling quasi-metric space. Then for every $\eta\in(1,\infty)$ there exist $\Lambda\in(\eta,\infty)$ and $M\in\N$, both depending on $\d$, $\eta$, and the geometric doubling constant of $({\mathbb{X}},\d)$, and which have the following significance.
	
	For each proper, nonempty, open subset $\Omega\subset {\mathbb{X}}$ there exists a sequence of points $\{c_j\}_{j\in\N}$ in $\Omega$ and a sequence of positive radii $\{r_j\}_{j\in\N}$, such that the following hold:
	\begin{enumerate}
		\item [\emph{(i)}] $ \Omega=\cup_{j\in\N} B(c_j,r_j)$.
		\item[\emph{(ii)}] We have that $\sum_{j\in\N} \ind_{B(c_j,\eta r_j)}\leq M.$
		\item[\emph{(iii)}] For each $j\in \N$ we have that $B(c_j,\eta r_j)\subset \Omega$ and $B(c_j,\Lambda  r_j)\cap ({\mathbb{X}}\setminus \Omega)\neq\varnothing$.
		\item[\emph{(iv)}] If $B(c_j,\eta r_j)\cap B(c_i,\eta r_i)\neq \varnothing$ for $i,j\in\N$ then $r_i\eqsim r_j$, with implicit constants independent of $i,j\in\N$.		
		\item[\emph{(v)}] The collection of balls $\{B(c_j,\frac15 r_j)\}_{j\in\mathbb N}$ is pairwise disjoint.
		\item[\emph{(vi)}] For every $j\in\N$ we have that $r_j = 2^{s_j}$ for some $s_j\in\R$.
	\end{enumerate}
\end{lemma}
%%%%%%%%%%%%%%%%%%%%%%%%%%%%%% LEMMA LEMMA LEMMA

%%%%%%%%%%%%%%%%%%%%%%%%%%%%%% PROOF PROOF PROOF
\begin{proof} The Whitney decomposition with properties (i)--(iv) is \cite[Theorem 2.4]{AlvMi}, verbatim; notice that (iv) above follows easily from (iii).  Property (v) follows from a well known reduction that we sketch here. Assume that we are given an $\eta$-Whitney cover $\{B(c_j,r_j)\}_{j\in \N}$ with $\eta>5$. Then by the $5R$-covering lemma, see for example \cite[Theorem 1.2]{Hein}, there exists a pairwise disjoint subcollection $\{B{c_{j_k},r_{j_k}}\}_{k\in\N}$ such that
	\[
	\Omega=\bigcup_{j\in\N} B(c_j,r_j) \subset \bigcup_{k\in\N} 5B(c_{j_k},r_{j_k}) \subset \Omega
	\]
by property (iii) and the assumption $\eta>5$. Now we can replace the original collection $\{B(c_j,r_j)\}_{j\in\N}$ with the collection $\{B(c_{j_k},5r_{j_k})\}_{k\in\N}$ and note that it satisfies properties (i)--(v). To see (vi) we just consider a $2\eta$-Whitney decomposition of some open set $\Omega$ with $\eta>5$ and replace $r_j$ by $2^{s_j}$ where $r_j < 2^{s_j}<2r_j$.
\end{proof}
%%%%%%%%%%%%%%%%%%%%%%%%%%%%%% PROOF PROOF PROOF
For future reference we record here an easy estimate for balls in the Whitney cover of an open set $\Omega$.

%%%%%%%%%%%%%%%%%%%%%%%%%%%%%% PROPOSITION PROPOSITION PROPOSITION
\begin{proposition}\label{prop:dist} There exist constants $b,D_1,D_2,D_3>1$ depending only on the quasi-metric constant $c_\d$ such that the following holds. For every $q\geq 1$, for every $\eta> 4(c_\d)^2 q$, and for every ball $L$ in the $\eta$-Whitney cover of some open set $\Omega$ we have for $x\in qL$ that
\[
\frac{\eta}{b}\leq \frac{\dist(x,\partial \Omega)}{r_L} \leq b\Lambda,
\]
and
\[
B\left(x,\frac{\dist(x,\partial \Omega)}{\Lambda}\right)\subset  D_1 L \subset D_2   B\left(x,\frac{\dist(x,\partial \Omega)}{\eta}\right)\subset  D_3 \frac{\Lambda}{\eta} L.
\]
In the displays above we have $\Lambda=\Lambda(\eta)$ as provided by the Whitney decomposition of Lemma~\ref{lem:whitney}.
\end{proposition}
%%%%%%%%%%%%%%%%%%%%%%%%%%%%%% PROPOSITION PROPOSITION PROPOSITION

Whenever we apply the Whitney decomposition above for some value of $\eta\in(1,\infty)$ in order to produce a covering $\{B_j\}_j$ of some open set $\Omega$ we will say that $\{B_j\}_j$ is an $\eta$-\emph{Whitney covering} of $\Omega$, and we will use properties (i)-(iv) above with no particular mention. 

%
% Another small reduction we will use is that, applying the $2\eta$-Whitney-type covering we can always assume that the radii $r_j$ of the Whitney covering are dyadic. This provides an $\eta$-Whitney cover of $\Omega$ with balls of dyadic radii. We will always use this reduction in what follows.

We note below the well known \emph{finite number of neighbors} property of the Whitney covering in the case that the metric space supports a doubling measure.
%%%%%%%%%%%%%%%%%%%%%%%%%%%%%% LEMMA LEMMA LEMMA
\begin{lemma}\label{l.finneigb} Let $({\mathbb{X}},\d,|\cdot|)$ be a space of homogeneous type, $\eta\in(1,\infty)$, and let $\{B\}_{B\in\mathcal B_{\Omega}}$ be an $\eta$-Whitney cover of an open set $\Omega$. Then for any $B\in\mathcal B_{\Omega}$ we have
	\[
	\sharp\{B' \in\mathcal B_{\Omega}:\, \eta B \cap \eta B' \neq \varnothing \} \lesssim_{\mathbb{X},\eta} M,
	\]
with the implicit constant depending on the doubling constant of $|\cdot|$ and the chosen $\eta$ of the Whitney cover.
\end{lemma}
%%%%%%%%%%%%%%%%%%%%%%%%%%%%%% LEMMA LEMMA LEMMA

%%%%%%%%%%%%%%%%%%%%%%%%%%%%%% PROOF PROOF PROOF
\begin{proof} Let $J_B\coloneqq \{B':\,\eta B' \cap\eta B \neq \varnothing \}$. Then
	\[
	\int \sum_{B' \in J_B}\ind_{\eta B'}=\sum_{B'\in J_B}|\eta B'|\eqsim_{\mathbb X} \sharp J_B |\eta B|
	\]
using (iv) of the Whitney decomposition together with the fact that $\eta B \cap \eta \B'\neq \varnothing $ for all $B' \in J_B$ and that $|\cdot|$ is doubling. On the other hand we have that $\cup_{B'\in J_B}\eta B' \subset c\eta B$ for some constant $c$ depending on ${\mathbb{X}}$, as $r_B\eqsim r_{B'}$, uniformly in $B'\in J_B$. We conclude that
\[
\sharp J_B |\eta B| \eqsim\int \sum_{B' \in J_B}\ind_{\eta B'} \leq M \Big|\bigcup_{B' \in J_B}\eta B' \Big|\leq M |c\eta B|\lesssim_{\mathbb X} M |\eta B|
\]
and the lemma follows.
\end{proof}
%%%%%%%%%%%%%%%%%%%%%%%%%%%%%% PROOF PROOF PROOF

In what follows we will need to split the support of our functions into essentially disjoint balls of fixed scale. This is done in the following lemma which follows by more or less standard arguments in spaces of homogeneous type. In fact it is essentially contained in \cite[Theorem 1.16]{Hein}.

%%%%%%%%%%%%%%%%%%%%%%%%%%%%%% LEMMA LEMMA LEMMA
\begin{lemma}\label{l.cover} Let $({\mathbb{X}},\d,|\cdot|)$ be a space of homogeneous type and let $B$ be a ball. For each $s\in\Z$ with $2^s\leq r_B$ there exists a finite collection of balls $\{L_\tau\}_\tau$ whose union covers $B$, with $r_{L_\tau}=2^s$, and such that $L_\tau\subset c_1B$ for each $\tau$, and for every $\rho>0$ we have $\sum_\tau \ind_{\rho L_\tau}\lesssim_{\rho,{\mathbb{X}}} 1$. The constant $c_1>0$ and the implicit constant depend on the homogeneous metric structure of ${\mathbb{X}}$ and on $\rho$.
\end{lemma}
%%%%%%%%%%%%%%%%%%%%%%%%%%%%%% LEMMA LEMMA LEMMA

%%%%%%%%%%%%%%%%%%%%%%%%%%%%%% PROOF PROOF PROOF
\begin{proof} Let $\mathcal B$ be the collection of balls $\{B(x,\frac{1}{5}2^s):\, x\in B\}$. Obviously this collection covers $B$ and $\sup_{B\in\mathcal B}r_B<\infty$. By the $5R$-covering lemma, see for example \cite[Theorem 1.2]{Hein}, there exists a disjoint subcollection $\mathcal B'=\{B_\tau\}_\tau \subset \mathcal B$ such that 
	\[
	\bigcup_{B\in\mathcal B} B\subset \bigcup_{B\in\mathcal B'}5B.
	\]
Here one can easily check that $\mathcal B'$ is necessarily finite. Let $L_\tau\coloneqq 5B_\tau$ for each $\tau$. Obviously we have that $L_\tau \subset c_1 B$ for some $c_1>0$ depending on the quasi-metric $\d$ as $r_{L_\tau}=2^s \leq r_B$. It remains to show the bounded overlap property. This follows by a well known argument that we include here for completeness. 

For $x\in {\mathbb{X}}$ and $\rho>0$ let $I_x\coloneqq \{\tau:\, x\in 5\rho B_\tau\}$. We have that $B(x,5^{-1}2^s)\subset cB_\tau\subset \tilde c B(x,5^{-1} 2^s)$ for all $\tau \in I_x$, where $c,\tilde c$ depend on the quasi-metric of ${\mathbb{X}}$ and on $\rho$. Since $|\cdot|$ is doubling and the balls $B(x,5^{-1}2^s) $ and $5\rho B_\tau$ intersect and have comparable radii we get that $|B(x,5^{-1}2^s)|\eqsim_\rho |B_\tau|$ for all $\tau \in I_x$, with implicit constants depending on the homogeneous metric structure of ${\mathbb{X}}$ and on $\rho$. Since the balls $B_\tau$ are disjoint we now have
\[
|B(x,5^{-1}2^s)|\gtrsim_{\rho,{\mathbb{X}}} \Big|\bigcup_{\tau \in I_x}B_\tau\Big|=\sum_{\tau\in I_x} |B_\tau| \gtrsim_{\rho,{\mathbb{X}}} \sharp I_x |B(x,5^{-1}2^s)|
\]
and since $|\cdot|$ is doubling we get $\sharp I_x\lesssim_{\rho,{\mathbb{X}}} 1$ uniformly in $x$, with implicit constants depending on the homogeneous metric structure of $({\mathbb{X}},\d,|\cdot|)$ and on $\rho$. Thus for $x\in {\mathbb{X}}$ we have
\[
\sum_\tau \ind_{L_\tau}(x)=\sum_{\tau} \ind_{5B_\tau}(x)=\sharp I_x\lesssim_{\rho,{\mathbb{X}}} 1
\]
and the proof is complete.
\end{proof}
%%%%%%%%%%%%%%%%%%%%%%%%%%%%%% PROOF PROOF PROOF

Finally we record a standard estimate for doubling measures that allows us to compare the ratio of radii of nested balls by the corresponding ratio of their measures; see for example \cite[(4.16)]{Hein}.

%%%%%%%%%%%%%%%%%%%%%%%%%%%%%% LEMMA LEMMA LEMMA
\begin{lemma}\label{l.ratio} Let $(\mathbb X,|\cdot|,\d)$ be a quasi-metric space of homogeneous type, that is, $|\cdot|$ is doubling. Then there exist constants $\Delta_{\mathbb X},\delta_{\mathbb X}>0$ depending only on the homogeneous metric structure of $\, \mathbb X$ such that for every pair of metric balls $B(x,r)\subset B(z,R)$ we have
	\[
	\frac{|B(x,r)|}{|B(z,R)|} \geq \Delta_{\mathbb X} \big(\frac{r}{R}\big)^{\delta_{\mathbb X}}.
	\]
In fact one can take $\Delta_{\mathbb X}\eqsim 1/\beta$ and $\delta_\mathbb X\eqsim \log_2 \beta$ with $\beta$ the doubling constant of $|\cdot|$ and the implicit constants depending on the quasi-metric constant of $\d$.
\end{lemma}
%%%%%%%%%%%%%%%%%%%%%%%%%%%%%% LEMMA LEMMA LEMMA

%%%%%%%%%%%%%%%%%%%%%%%%%%%%%% SECTION SECTION SECTION
\subsection{Sparse collections of stopping balls} In this subsection we employ the Whitney decomposition  in order to construct stopping collections of metric balls associated with a given pair of functions. As the underlying measure is doubling,  the maximal operators
\[
\mathrm{M}_p f \coloneqq \sup_{B} \l f  \r_{p,B}\ind_{B},\qquad p\geq 1,
\]
with the supremum take over all metric balls in $\mathbb X$, satisfy
\[
\|\mathrm{M}_p\|_{p\to p,\infty} \lesssim_{ {\mathbb{X}}} 1, \qquad \|\mathrm{M}_p\|_{q\to q} \lesssim_{{\mathbb{X}},q} 1, \quad q>p\geq 1.
\]
A local version  of the maximal function $\mathrm{M}_p$ is introduced as follows. Given $\Delta>1$,  a non-empty open set $\Omega\subset \mathbb X$, $f\in L^p_{\mathrm{loc}} (\mathbb X), $ 
\[
\mathrm{M}_p ^{\Omega,\Delta} f(x) \coloneqq \sup_{B:\, \mathrm{dist}(B,\partial \Omega)\geq \Delta  r_B} \l  f  \r_{p,B}\ind_{B}(x), \qquad x\in \Omega.
\]
Note that all the balls contributing in the supremum defining $\mathrm{M}_p ^\Omega$ are well inside $\Omega$. We now construct the sparse collection of stopping balls that will be use for the proofs of our main results.

%%%%%%%%%%%%%%%%%%%%%%%%%%%%%% LEMMA LEMMA LEMMA
\begin{lemma}\label{lem:tedious} Let $K>1$ be a positive integer, $1\leq p_1,p_2<\infty$ and for $i\in\{1,2\}$ let $f_i\in L^{p_i} _{\mathrm{loc}}(\mathbb X)$ be a pair of functions supported in $c_oB_0$. For every $q>1$ sufficiently large there exist a constant $c_1>1$ depending only on $c_o,q$ and the homogeneous structure of $\, \mathbb X$, open sets \[E_K\subset E_{K-1}\subset \cdots \subset E_1\subset E_0\coloneqq c_o B_0\] and collections of metric balls $\mathcal B_1,\ldots,\mathcal B_K$ with the following properties.
\begin{enumerate}
\item [\emph{(i)}]  $\mathcal B_0\coloneqq \{c_oB_0\}$, and for $k\geq 1$ each collection $\mathcal B_k$ is a $q$-Whitney cover of $E_k$. In particular $\mathcal B_k$ satisfies properties \emph{(i)--(vi)} of Lemma~\ref{lem:whitney} with $\eta=q$.  Furthermore $q$ has the property \eqref{eq:qcont} below.
\item[\emph{(ii)}] Denoting
\[
\mathcal B_k(B)\coloneqq\{L\in \mathcal B_k: c_o L\cap B\neq \varnothing\}, \qquad B\in \mathcal B_{k+1},\quad k\in\{0,\ldots,K-1\},
\]
we have 
$r_B\leq \frac12 r_L$ for all $L \in \mathcal B_{k}(B)$. In particular
\[
\sup_{B\in\mathcal B_{k+1}} r_{B}\leq \frac12 \sup_{B\in\mathcal B_{k}}r_{B}.
\]
Furthermore we have that $|E_{k+1}|\leq |E_k|/2$ for all $k\geq 0$.  
\item[\emph{(iii)}] There exists $\zeta=\zeta(\mathbb X)>0$ such that the collection $\bigcup_{k=0} ^K \mathcal B_k$ is $\zeta$-sparse.	
\item[\emph{(iv)}] For every $k\in\{0,\ldots,K-1\}$ and $x\in E_k\setminus E_{k+1}$ there holds
\[
|f_i(x)| \lesssim_{\mathbb X}  \inf_{\substack{B\in\mathcal B_k\\q B\ni x}} \l  f_i  \r_{p_i,c_1B}, \qquad i=1,2.
\]
\item[\emph{(v)}] For every $k\in\{0,\ldots,K-1\}$ $B\in\mathcal B_{k+1}$, $L\in \mathcal B_{k}(B)$, there holds 
\[
\langle  f_i  \rangle_{p_i,qB}\lesssim_{\mathbb X}  \langle  f_i  \rangle_{p_i,c_1 L}, \qquad i=1,2.
\]
\end{enumerate}
\end{lemma}
%%%%%%%%%%%%%%%%%%%%%%%%%%%%%% LEMMA LEMMA LEMMA

%%%%%%%%%%%%%%%%%%%%%%%%%%%%%% PROOF PROOF PROOF
\begin{proof} During the course of this proof, the  following positive constants will be so chosen.

\begin{itemize}[leftmargin=5.5mm]
\item[$\cdot$]  Remember that we denote by $c_\d$  the quasimetric constant of $\,\mathbb X$ and by $c_o$ the localization constant in \eqref{eq:localization}. 
 \item[$\cdot$]We will impose the condition that $q>1$ is sufficiently large, depending only on $c_o$ and the quasi-metric constant $c_\d$, so that the following holds: If $L,B$ are two metric balls then
 \begin{equation}\label{eq:qcont}
 c_o L \cap B \neq \varnothing\quad\text{and}\quad L\nsubseteq qB \implies B\subset qL.
 \end{equation}
In fact taking  $q\geq 10c_\d ^2 c_o$ will suffice. 
\item[$\cdot$] The constants $b,D_1,D_2,D_3$ from Proposition~\ref{prop:dist}, with input the (sufficiently large) value of $q$ of the statement, will be used throughout the proof. A Whitney parameter $\eta$  will be chosen to be sufficiently large in $[ 4(c_\d)^2 b q,\infty )$ and $\Lambda =\Lambda(\eta)$ is the corresponding value from  Lemma~\ref{lem:whitney}. 
\item[$\cdot$] $\Theta$ is a large constant depending on the weak-type inequality for the geometric maximal operator $\mathrm{M}_1$ and thus ultimately on $\mathbb X$, and on $\eta$. 
%\item[$\cdot$] The parameter $\kappa\in (0,1)$ will be used internally in the proof. This parameter is chosen sufficiently small  depending on the space $\mathbb X$ and $c_o$ and meant to be  fixed throughout.
\end{itemize}

Before we begin the proof we notice that for (iii) it will be enough to prove:
\begin{itemize}
	\item [(iii$^\prime$)]  There exists $\zeta=\zeta(\mathbb X)>0$ such that for each $k\geq 0$ and each $B\in\mathcal B_k$ there exists $E_B\subset E_{k}\setminus E_{k+1}$ with $|E_B|>\zeta |B|$, and the sets $\{E_B\}_{B\in\mathcal B_k}$ are pairwise disjoint.
\end{itemize}
Clearly (iii$^\prime$) will imply (iii). Indeed, for fixed $k$ the sets $\{E_B\}_{B\in\mathcal B_k}$ are pairwise disjoint and $|E_B|>\zeta|B|$ for all $B\in\mathcal B_k$ because of property (iii$^\prime$). Furthermore
	\[
	\bigcup_{B\in\mathcal B_k}E_B \subset E_k\setminus E_{k+1}
	\]
and the collection $\{E_k\setminus E_{k+1}\}_{k\geq 1}$ is clearly pairwise disjoint since the sequence of sets $\{E_k\}_k$ is decreasing.

Let us fix a parameter $\eta>q$ 
%and $\kappa\in(0,1)$ 
as described above. Having already set $E_0\coloneqq c_o B_0$ and $\mathcal B_0\coloneqq \{c_o B_0\}$, start the proof proper by defining
	\[
	E_1\coloneqq \bigcup_{i=1,2}\Big\{x\in c_o B_0:\,  \mathrm{M}_{p_i} f_i(x) > 2^\Theta \l f_i\r_{p_i,c_oB_0}\Big\}.
	\]
%where $c_1,\Theta>1$ are structural constants to be specified in the course of the proof.
Then, let $\mathcal B_1$ be the $\eta$-Whitney cover of $E_1$ provided by Lemma~\ref{lem:whitney}. Note that $\eta B\subset E_1\subset E_0=c_oB_0$ for all $B\in\mathcal B_1$ by the properties of the Whitney decomposition. Since $\eta>c_o$ this proves (i) for the base step of our inductive construction.

Next is the verification of    properties (ii) to (v) in the base step $k=0$.
In order to prove (ii) note that the maximal theorem provides the estimate
\[
|E_1|\lesssim_{\mathbb{X}}  2^{-\Theta}  |B_0| 
\]
and thus   choosing $\Theta$ sufficiently large we will have that $|E_1|\leq |B_0|/2 \leq |E_0|/2$.

Since $B\subset  B_0$ for every $B\in\mathcal B_1$, Lemma~\ref{l.ratio} implies that 
\[
\Big(\frac{r_B}{ r_{B_0}}\Big)^{\delta_{\mathbb{X}}} \lesssim_{\mathbb{X}}  \frac{| B |}{| B_0|}\leq\frac{|E_1|}{|B_0|}\lesssim  2^{-\Theta}
\]
and thus (ii) follows if $\Theta$ is chosen sufficiently small depending on the homogeneous structure of $\,\mathbb X$.
For property (iii$^\prime$)  we define $E_{B_0}\coloneqq B_0\setminus E_1$. As  $|E_1|\leq  2^{-1}|B_0|$ by our previous choice 
\[
E_{B_0}\subset B_0\setminus E_1,\qquad |E_{B_0}| \geq |B_0|-|E_1|\geq 2^{-1}|B_0| 
\]
as desired.

To prove (iv), note that  $x\in E_0\setminus E_1=c_oB_0\setminus E_1$ the Lebesgue differentiation theorem yields
\[
|f_i(x)|\leq \mathrm{M} ^{B_0} _{p_i}f_i(x)\leq 2^\Theta \l  f_i  \r_{p_i,c_oB_0}.
\]
To prove (v), notice that by the properties of the Whitney decomposition we may find $x\in qB\in \mathcal B_1$ and $y\in E_0\setminus E_1=c_oB_0\setminus E_1$ with $\d(x,y)\leq \Lambda r_{B}$. It is easy to check that the ball $B'$ of radius $2c_\d\Lambda r_B$ and center $x$ contains $qB$ and $y\in B'$ as well. Therefore  for $i=1,2$ we will have
\[
 \langle  f_i \rangle_{p_i, qB} \lesssim_{\mathbb X}   \langle  f_i \rangle_{p_i, B'} \leq     \mathrm{M}_{p_i} f_i(y) \lesssim   \langle| f_i|\rangle_{p_i,c_oB_0}
\]
thus concluding the proof of (v) and the treatment of the $k=0$ case.

Now assume inductively that that $(E_1,\mathcal B_1),\ldots, (E_k,\mathcal B_{k})$ with the desired properties have been constructed,  and proceed with the construction of $(E_{k+1},\mathcal B_{k+1})$. To that end let
\begin{equation}
\label{e:chdelta}
E_{k+1}\coloneqq \bigcup_{i=1,2} \bigg\{x\in E_k:\, \mathrm{M}_{p_i} ^{E_k,\Delta}f_i(x)  >  2^\Theta  \l  f_i \r_{p_i,D_2B\left (x,\frac{\d(x,\partial E_k)}{\eta}\right)}  \bigg\}, \qquad \Delta\coloneqq 2c_\d\Lambda,
\end{equation}
and we let $\mathcal B_{k+1}$ be the $\eta$-Whitney collection associated with the open set $E_{k+1}$. Since $\eta>q$ we have that (i) is automatically satisfied. 
In order to prove (ii) let $L\in\mathcal B_{k}$ and $x\in q L \cap E_{k+1}$.  Proposition~\ref{prop:dist} tells us that
\begin{equation}
\label{e:supradius}
\begin{split}
&\frac{\eta}{b}\leq \frac{\d(x,\partial E_k)}{r_L} \leq b\Lambda\qquad\text{and}
\\
& B(x) \coloneqq  B\left (x,\frac{\d(x,\partial E_k)}{\Lambda}\right) \subset D_1 L\subset D_2  B\left (x,\frac{\d(x,\partial E_k)}{\eta}\right)\subset D_3\frac{\Lambda}{\eta}L.
\end{split}
\end{equation}
 We will have for either $i=1$ or $i=2$ that
\begin{equation}
\label{e:sup}
\l  f_i  \r_{p_i,D_1 L} \lesssim_{\mathbb X,\eta}  \l  f_i \r_{p_i,D_2B\left (x,\frac{\d(x,\partial E_k)}{\eta}\right)}  <  2^{-\Theta} \sup_{\substack{B'\ni x\\ \d(B',\partial E_k)>\Delta r_{B'}}} \l  f_i  \r_{p_i,B'};
\end{equation}
the second inequality is the membership $x\in E_{k+1}.$
The choice of $\Delta$ in \eqref{e:chdelta} ensures that  the balls appearing in the above supremum   are all contained in $B(x)$, and \emph{a fortiori} in  $D_1 L$. This entails
\[
 q L\cap E_{k+1} \subset \bigcup_{i=1,2}\Big\{x\in\mathbb X:\, \mathrm M_{p_i}(f_i\ind_{D_1L})(x)  \gtrsim_{\mathbb X,\eta}  {2^\Theta}  \l  f_i  \r_{p_i,D_1L}\Big\}
\]
so that, by the maximal theorem, we get the estimate $| q L\cap E_{k+1}|\lesssim 2^{-\Theta} |L|$, so that
\begin{equation}
\label{e:reduce}
| q L\cap E_{k+1}|\leq 2^{-\frac{\Theta}{2}} |L| 
\end{equation}
provided that $\Theta$ is chosen large enough depending only on $\mathbb X$ and $\eta$. Summing over $L\in\mathcal B_k$ and using the finite overlap of the balls in $\mathcal B_k$ and the fact that they cover $E_{k+1}$ yields $|E_{k+1}|\lesssim 2^{-\frac{\Theta}{2}}|E_k|$ and so $|E_{k+1}|\leq |E_k|/2$ if $\Theta$ was chosen sufficiently large.

Consider now any $B\in \mathcal B_{k+1}$ and recall that $\mathcal B_k(B)=\{L\in\mathcal B_k:  c_oL\cap B\neq \varnothing.\}$ As $B\subset E_{k+1}\subset E_k$ and $E_k$ is covered by the balls in $\mathcal B_k$, $\mathcal B_k(B)$ is not an empty collection. 
If $L\in \mathcal B_k(B) $, then $L\nsubseteq q B$ otherwise, since $q B\subset \eta B\subset  E_{k+1}$, the contradiction
%We claim that for all such $L$ we have $r_L>c(\eta)r_B$ where the constant $c(\eta)<1$ depends only on $\eta$ and the homogeneous structure of $\, \mathbb X$. In order to see this we first show that $L\nsubseteq \eta B$. Indeed, arguing by contrdiction, if $L\subset \eta B$ then we would have
\[
|L |=|  L \cap q B|\leq |L \cap E_{k+1}|\leq  2^{-\frac{\Theta}{2}} |L|
\]
is reached in view of the containment  $q B\subset  E_{k+1}$ and \eqref{e:reduce}. Therefore $L\nsubseteq qB$ and $c_oL\cap B\neq \varnothing$ which by the definition of $q$ implies that $qB\subset q L$ and thus
$
 | qB| \leq  |q L\cap E_{k+1}|\leq   2^{-\frac{\Theta}{2}}  | L|
$, again by  \eqref{e:reduce}. 
The last two observations and an application of Lemma~\ref{l.ratio}  yields 
\begin{equation}
\label{e:kappause}
L\in \mathcal B_k(B)\implies
r_B\leq \beta_12^{-\frac{\Theta\beta_2}{2}} r_L
\end{equation}  which yields (ii) if $\Theta$ was chosen sufficiently large depending on the homogeneous structure of $\mathbb X$.
Finally, for each $B=B(c_B,r_B)\in\mathcal B_k$ we set $E_B\coloneqq B(c_B,\frac15 r_B)\setminus E_{k+1}$. Clearly $E_B\subset B\setminus E_{k+1}\subset E_k\setminus E_{k+1}$ and in addition, 
\[
|E_B|\geq |B(c_B,\frac15 r_B)|- |B\cap E_{k+1}|\geq c_{\mathbb X}^{-1} |B|-2^{-\frac{\Theta}{2}} |B| \geq \zeta |B|. \]
 In the second inequality  $c_{\mathbb X}>1$ is some   some structural constant dependent on by the doubling property of the measure $|\cdot|$,  and \eqref{e:reduce} has been used. Then     (iii$^\prime$) has been achieved, provided that $\Theta$ is sufficiently large and $\zeta$ is sufficiently small, depending only on $\mathbb X$.

The argument for (iv) is as follows. Let $x\in c_o L\setminus E_{k+1}\subset qL\setminus E_{k+1}$ for $L\in \mathcal B_k$. Define
\[
B'(x)\coloneqq D_2B\left (x,\frac{\dist(x,\partial E_k)}{\eta}\right)\subset c_1 L,\qquad c_1\coloneqq\Lambda D_3/\eta
\]
with $D_3$ as in Proposition~\ref{prop:dist}.
Using the Lebesgue theorem for the first inequality and the definition of the set $E_{k+1}$ for the second, we get
\[
|f_i(x)| \leq  2^\Theta   \l  f_i \r_{p_i,B'(x)}   \lesssim  \l  f_i \r_{p_i,c_1L}, \qquad i=1,2.
\]

Turning to the proof of (v), fix $B\in \mathcal B_{k+1}$ and $L\in \mathcal B_k(B)$, $x_L\in B\cap L$. By the Whitney property of $B$ we may find  $y\in \Lambda B\setminus E_{k+1}$ so that by the triangle inequality $\d(x_L,y)\leq c_\d(\Lambda+c_\d) r_B\leq 2c_\d \Lambda r_B$. By another application of the triangle inequality, it follows that 
\[\d(c_L,y)\leq c_\d \d(x_L,y) + c_\d r_L \leq   2(c_\d)^2 \Lambda r_B +c_\d r_L \leq q r_L
\]
 provided that $\Theta$ is chosen large enough, in virtue of \eqref{e:kappause}. Therefore $y\in q L\setminus E_{k+1} $, and since $L\in\mathcal B_k$ Proposition~\ref{prop:dist} yields
\[
\frac{1}{b}\eta \leq \frac{\dist(y,\partial E_k)}{r_L}\leq b \Lambda.
\]
Let $z\notin E_k$ such that $\dist(y,z)\geq b^{-1}\eta r_L$ and let $h\in \Lambda B$. Then by the triangle inequality
\[
\d(h,z)\geq c_\d ^{-1} b^{-1} \eta r_L - 2 c_\d \Lambda r_B - 2qc_\d r_L >2c_\d \Lambda r_B
\]
provided that $\Theta $ is large enough. Thus the ball $\Lambda B\supset \eta B\supset q B$ satisfies $\dist(\Lambda B,\partial E_k)>\Delta r_B$ and also contains the point $y$. Since $\Lambda>\eta>1$ we get for $j=1,2$
\[
 \langle  f_i \rangle_{p_i,q B} \lesssim_{\mathbb X}    \langle  f_i \rangle_{p_i, \Lambda B} \leq   \mathrm{M}_{p_i}^{E_k,\Delta} f_i(y) \lesssim   \langle| f_i|\rangle_{p_i,c_1 L}, \qquad i=1,2
\]
completing the proof of (v) and in turn of the Lemma.
\end{proof}
%%%%%%%%%%%%%%%%%%%%%%%%%%%%%% PROOF PROOF PROOF

%%%%%%%%%%%%%%%%%%%%%%%%%%%%%% SECTION SECTION SECTION
\section{Proof of Theorem \ref{th:SD}}\label{sec:proofmain} Throughout the section we fix a space of homogeneous type $({\mathbb{X}},\d,|\cdot|)$.

%%%%%%%%%%%%%%%%%%%%%%%%%%%%%% SECTION SECTION SECTION
\subsection{Stopped forms} \label{ss:sf}% We now consider an operator $T$ of form

Let $c_o$ be the constant in \eqref{eq:localization}, $L\subset\mathbb X$ be a ball of radius $r_L$, and  $\mathcal  B$ be a Whitney collection with the property that 
\begin{equation}\label{eq:stopwhit}
B\in \mathcal B, \, c_o L \cap B \neq \varnothing \implies r_B \leq r_L/2,\quad B\subseteq qL.
\end{equation}
We set $E\coloneqq \bigcup \{B:B\in \mathcal B\}$ and for $1\leq p<\infty$, $h\in L^\infty(\mathbb X)$, introduce the $p$-stopping norm with data $(L,\mathcal B)$ by
\begin{equation}
\label{e:stop} \|h\|_{p,(L,\mathcal B)}\coloneqq \left\|h \ind_{c_o L\setminus E}\right\|_\infty + \sup_{B\in 
\mathcal B}  \l  h  \r_{p,B}.
\end{equation}

%%%%%%%%%%%%%%%%%%%%%%%%%%%%%% REMARK REMARK REMARK
\begin{remark}\label{rmrk:lp<stopped} The stopping norm controls the local $L^p$ norms, in the following sense. Notice that the balls $ \{B\in \mathcal B:B\cap c_o L\neq \varnothing\}$ cover $c_o L\cap E$, are contained in $ q L$ and have bounded overlap. Therefore
\[\begin{split}
|L|\l  h \ind_E\r_{p,c_o L}^p  &\leq   \sum_{\substack{ B\in \mathcal B\\ B\cap c_o L\neq \varnothing} }|B| \l h \r_{p,B}^p   
\leq  \left(\sup_{B\in 
\mathcal B}  \l h \r_{p,B}^p\right)  \sum_{B\in \mathcal B_L} |B| 
  \lesssim |L|   \|h\|_{p,(L,\mathcal B)}^p.
\end{split}
\]
As $\l h\ind_{c_o L\setminus E} \r_{p,c_o L} \leq  \|h\|_{p,(L,\mathcal B)}$, it follows that
\begin{equation}
\label{e:stopbd1}  \langle h\ind_E\rangle_{p,c_o L} \lesssim \|h\|_{p,(L,\mathcal B)}.
\end{equation}
\end{remark}
%%%%%%%%%%%%%%%%%%%%%%%%%%%%%% REMARK REMARK REMARK

%%%%%%%%%%%%%%%%%%%%%%%%%%%%%% REMARK REMARK REMARK
\begin{remark}\label{rmrk:stopped<lp} Let $\{\mathcal B_k\}_k$ denote the Whitney collections constructed in Lemma~\ref{lem:tedious} for a pair of functions $f_1,f_2$ supported in $c_o B_0$ and $q>c_o$, and let $ L=B_k\in\mathcal B_k$ and $\mathcal B=\mathcal B_{k+1}$ for some $k\geq 0$. We then know by property (ii) of the lemma that if $c_o B_k \cap B \neq \varnothing$ for $B\in\mathcal B_{k+1}$ we will have that $r_L \leq r_B/2$ so the setup of \eqref{eq:stopwhit} applies. Then the estimate of Remark~\ref{rmrk:lp<stopped} can be reversed in the following sense:
\[
\begin{split}
\| f_i \ind_{c_o B_k }\|_{p_i,(B_k,\mathcal B_{k+1})} &\leq \left\|f_i \ind_{c_oB_k\setminus E_{k+1}}\right\|_\infty + \sup_{B\in 
\mathcal B_{k+1}}  \langle f_i \ind_{c_oB_k} \rangle_{p_i,B} \lesssim \l f \r_{p_i,c_1 B_k}.
\end{split}
\]
The estimate for the first summand follows by the fact that  $ c_o B_k \setminus E_{k+1}\subset qB_k\setminus E_{k+1}\subset  E_k\setminus E_{k+1}$ and a use of (iv) of Lemma~\ref{lem:tedious}. For the second summand above note that all the balls $B\in\mathcal B_{k+1}$ that participate in the supremum must intersect the ball $c_o B_k$, namely we have that $B_k\in \mathcal B_k(B)$. The estimate then follows by property (v) of Lemma~\ref{lem:tedious}. 
\end{remark}
%%%%%%%%%%%%%%%%%%%%%%%%%%%%%% REMARK REMARK REMARK
We now consider an operator $T$ of form
\[
T_{\sigma} ^\tau f (x)= \sum_{\sigma\leq s <\tau} [T(s)f](x), \qquad x\in\mathbb X,\quad \sigma<\tau,
\]
satisfying \eqref{eq:localization}. For $L,\mathcal B$ as above, a  partition of unity $\{\phi_B:B\in \mathcal B\}$ subordinate to $\mathcal B$, and for $\sigma> s_L$, we define the stopping form
\begin{equation}
\label{e:stofo}
\Lambda^\sigma _{(L,\mathcal B)}(h_1,h_2)\coloneqq  \left\langle T_\sigma ^{s_{L}}\left[ h_1\ind_{L\setminus E}\right],  h_2 \right\rangle  +\sum_{B\in\mathcal B}\left\langle  T_{ s_{B}\vee \sigma} ^{s_{L}}\left[h_1 \ind_{L} \phi_{B}\right], h_2\right\r.
\end{equation}
where $a\vee b\coloneqq \max\{a,b\}$.
%%%%%%%%%%%%%%%%%%%%%%%%%%%%%% LEMMA LEMMA LEMMA
\begin{lemma}\label{lemma:iter} Let  $T$ satisfy the assumptions of Theorem~\ref{th:SD} and $L,\mathcal B$ be as in \eqref{eq:stopwhit}. Then
\[
\sup_{\sigma>0}\big|\Lambda^\sigma _{(L,\mathcal B)}(h_1,h_2)\big| \lesssim_{X} \left[C_{ p }+  \|\omega\|_{\mathrm{Dini}}\right]|L| \|h_1\|_{p_1,(L,\mathcal B)} \|h_2\|_{p_2,(L,\mathcal B)}.
\]
\end{lemma}
%%%%%%%%%%%%%%%%%%%%%%%%%%%%%% LEMMA LEMMA LEMMA

The proof of Lemma \ref{lemma:iter} is postponed to the next subsection. We now show how to prove Theorem \ref{th:SD} with a combination of Lemma \ref{lemma:iter} and the main decomposition of Lemma \ref{lem:tedious}. This is done as follows.

%%%%%%%%%%%%%%%%%%%%%%%%%%%%%% SECTION SECTION SECTION
\subsection{Compiling the proof of Theorem~\ref{th:SD}}\label{s.together} In this subsection we put together all the pieces needed for the complete proof of Theorem~\ref{th:SD}. We explain in detail the steps needed to prove the conclusion for the untruncated version $T\sim \sum_s T(s)$; the proof for the truncated version is similar but simpler.

Let $T\sim\sum_s T(s)$ be a linear operator satisfying the assumptions of the theorem and $f_1,f_2$ Lipschitz functions supported on some ball $B_0$ with $r_{B_0}=2^{s_{B_0}}$. By assumption, there exists $p\in(1,\infty)$ such that $\sup_{\sigma<\tau}\|T_\sigma ^\tau\|_{L^{p}\to L^{p}}=C_{p}<\infty$. By Remark~\ref{rmrk:represent} we have
\[
\l Tf_1,f_2 \r= \l m f_1,f_2 \r + \lim_{\substack{j\to \infty}}\l T_{\sigma_j} ^{\tau_j} f_1,f_2\r
\]
for some $m\in L^\infty$. 

%%%%%%%%%%%%%%%%%%%%%%%%%%%%%% SECTION SECTION SECTION
\subsection{First term} The estimate for this term is very easy. Applying Lemma~\ref{lem:tedious} for Lipschitz functions $f_1,f_2$ supported in some ball $B_0$ we will have that
\[
\begin{split}
|\l mf_1,f_2\r|& \leq \|m\|_\infty \sum_{k=0} ^K \l f_1 \ind_{E_k\setminus E_{k+1}}, f_2 \r + \|m\|_\infty\int_{E_{K+1}}|f_1||f_2|
\\
& \leq \|m\|_{L^\infty}   \sum_{B\in\mathcal B_k}  \sum_{k=0} ^K \l f_1 \phi_{B_k},f_2 \ind_{B_k}\r + \|m\|_\infty \|f_1\|_\infty \|f_2\|_\infty |E_{K+1}|.
\end{split}
\]
Note that the first summand above provides a sparse form. Indeed by (iv) of Lemma~\ref{lem:tedious} we will have that
\[
 \sum_{B\in\mathcal B_k}  \sum_{k=0} ^K \l f_1\phi_{B_k},f_2 \ind_{B_k}\r \lesssim  \sum_{B\in\mathcal B_k}  \sum_{k=0} ^K |B_k|\l f_1\r_{c_1B_k}\l f_2\r_{c_1B_k}
\]
and the collection $\cup_{k=0} ^K \mathcal B_k$ is sparse. For the second summand we notice that by (ii) of Lemma~\ref{lem:tedious} we have that $|E_K|\leq 2^{-K}|B_0|$ and we choose $K$ sufficiently large so that  
\[
\|m\|_\infty \|f_1\|_\infty \|f_2\|_\infty |E_{K+1}|<|\l mf_1,f_2\r|/2<\infty,
\]
and absorb this term in the left hand side. Remembering from Remark~\ref{rmrk:represent} that $\|m\|_\infty\lesssim C_p$ completes the proof for the first term.
%%%%%%%%%%%%%%%%%%%%%%%%%%%%%% SECTION SECTION SECTION
\subsection{Second term} For the second term we now choose $\sigma_o,\tau_o\in\Z$ with $\sigma_o<\tau_o$ such that
\[
 \lim_{j\to \infty}|\l T_{\sigma_j} ^{\tau_j} f_1,f_2\r|\leq 2 |\l T_{\sigma} ^{\tau} f_1,f_2\r|,\qquad \forall \sigma\leq \sigma_o,\, \tau\geq \tau_o.
\]
Note that it is without loss of generality to assume  $s_{B_0}\geq \tau_o$ by taking a bigger ball $B_0 '\supset B_0$, if necessary. Thus it suffices to estimate
\[
|\l T_{\sigma} ^{s_{B_0}} f_1,f_2\r|=|\l T_{\sigma} ^{s_{B_0}} (f_1\ind_{B_0}),\ind_{c_oB_0}f_2\r|,\qquad  \sigma \leq \sigma_o<s_{B_0},
\]
where we used the localization principle \eqref{eq:localization}.

We then apply Lemma \ref{lem:tedious} to $f_1,f_2,$ supported in $c_o B_0$ and $q=q(c_o)>c_o$ sufficiently large so that \eqref{eq:qcont} is satisfied together with the other conclusions of the lemma. We obtain Whitney collections $\mathcal B_0=\{c_oB_0\},\ldots, \mathcal B_{K+1}$ such that $ \mathcal S=\mathcal B_0\cup \cdots \cup \mathcal B_K$ is a $\zeta$-sparse collection. Property (ii) of Lemma~\ref{lem:tedious} guarantees that for every $k=0,2,\ldots,K-1$ we will have 
\[
B\in\mathcal B_{k+1} \implies r_B \leq \frac12 r_L\qquad \forall L\in\mathcal B_k(B)\coloneqq\{L\in\mathcal B_{k}:\, c_oL\cap B\neq \varnothing\}.
\]
In particular, by choosing the  ending parameter $K+1$ sufficiently large we can guarantee that 
\[
\sup\{r_B: B\in \mathcal \mathcal B_{K+1}\} \leq 2^{-K} c_or_{B_0}<\sigma.
\]

Associate to each $\mathcal B_k=\{B_k\}$ a partition of unity $\{\phi_{B_k}:B_k\in \mathcal B_k\}$ on $E_k$, and let $\mathcal B_k^\sigma=\{B\in \mathcal B_k: s_B>\sigma\}$. Notice that in the case $k=0$ 
we simply have $\phi_{B_0}=\ind_{B_0}$. We claim the following equality obtained inductively for all $\ell\leq K$
\begin{equation}
\label{e:maindeco2}
\l T_{\sigma} ^{s_{B_0}} f_1,f_2\r = \sum_{k=0}^\ell
 \sum_{B_k\in \mathcal B_k^\sigma} \Lambda^\sigma_{(B_k, \mathcal B_{k+1})} (f_1\phi_{B_k},f_2 ) + \sum_{B_{\ell+1}\in \mathcal B_{\ell+1}^\sigma}
\l T_\sigma ^{s_{B_{\ell+1}}} [f_1    \phi_{B_{\ell+1}}],f_2  \r,
\end{equation}
where $\Lambda^\sigma_{(B_k, \mathcal B_{k+1})}$ refers to the stopping form \eqref{e:stofo} with $L=B_k,\mathcal B=\mathcal B_{k+1}$: the proof of \eqref{e:maindeco2} is postponed until the end of the section. In particular $\mathcal B_{K+1}^\sigma$ is empty, and we obtain
\begin{equation}
\label{e:maindeco3}
 \l T_{\sigma} ^{s_{B_0}} f_1,f_2\r = \sum_{k=0}^K
 \sum_{B_k\in \mathcal B_k^\sigma} \Lambda^\sigma_{(B_k, \mathcal B_{k+1})} (f_1\phi_{B_k},f_2\ind_{c_o B_k} )  
\end{equation}
where we also used the support of $T^{s_{B_k}}_\sigma$ to insert the cutoffs $ \ind_{c_o B_k}$ on $f_2$.  We then estimate for $B_k\in \mathcal B_k^\sigma$, using Lemma \ref{lemma:iter}
\begin{equation}
\label{e:maindeco4}
\begin{split}
 \big|\Lambda^\sigma_{(B_k, \mathcal B_{k+1})} (f_1\phi_{B_k},f_2\ind_{q B_k})\big| & \leq  \left[C_{ p }+  \|\omega\|_{\mathrm{Dini}}\right]|B| \|f_1\|_{p_1,(B_k, \mathcal B_{k+1})} \|f_2\|_{p_2,(c_oB_k, \mathcal B_{k+1})}\\ &\lesssim  \left[C_{ p }+  \|\omega\|_{\mathrm{Dini}}\right] |B_k|\l f_1\r_{p_1,c_1 B_k} \l f_2\r_{p_2,c_1 B_k},
\end{split}
\end{equation}
where we used Remark~\ref{rmrk:stopped<lp} to pass to the last line.  Combining \eqref{e:maindeco3} with  \eqref{e:maindeco4} leads to 
\[
\big|\l T_{\sigma} ^{s_{B_0}} f_1,f_2\r\big| \lesssim  \left[C_{ p }+  \|\omega\|_{\mathrm{Dini}}\right] \sum_{B\in \mathcal S} |B|\l f_1\r_{p_1,c_1 B} \l f_2\r_{p_2,c_1 B}
 \]
which is the sparse estimate being sought.

%%%%%%%%%%%%%%%%%%%%%%%%%%%%%% PROOF PROOF PROOF
\begin{proof}[Proof of \eqref{e:maindeco2}] The proof is by induction on $\ell$.
For all $k=0,\ldots,K$ and $B_k\in \mathcal B_k^\sigma$, we have
\begin{equation} \label{e:trickster}
\l T_\sigma^{s_{B_k}} [f_1 \phi_{B_k}],f_2  \r  =  \Lambda^\sigma_{(B_k, \mathcal B_{k+1})}(f_1\phi_{B_k}, f_2 ) + \sum_{B_{k+1}\in \mathcal B_{k+1}^\sigma}
\l T_\sigma ^{s_{B_{k+1}}} [f_1 \phi_{B_k}  \phi_{B_{k+1}}],f_2  \r.
\end{equation}
%We have in particular used the support of $T^{s_{B_k}}_\sigma$ to insert the cutoffs $ \ind_{q B_k}$, $ \ind_{q B_{k+1}}$ on $f_2$. 
In particular \eqref{e:maindeco2} holds for $\ell=0$ as $\phi_{B_0}=1$ on the support of $f_1$. Suppose now $\ell\geq 1$ and \eqref{e:maindeco2} has been verified for $\ell-1$ in place of $\ell$. Using \eqref{e:trickster} on each $B_\ell\in \mathcal B_\ell$ 
\[
\begin{split} &\quad
\langle T_{\sigma} ^{s_{B_0}} f_1,f_2\r - \sum_{k=0}^{\ell}
 \sum_{B_k\in \mathcal B_k^\sigma} \Lambda^\sigma_{(B_k, \mathcal B_{k+1})} (f_1\phi_{B_k},f_2) = \sum_{B_\ell \in \mathcal B_{\ell}^\sigma}  \sum_{B_{\ell+1}\in \mathcal B_{\ell+1}^\sigma}
\l T_\sigma ^{s_{B_{\ell+1}}} [f_1 \phi_{B_\ell}  \phi_{B_{\ell+1}}],f_2  \rangle
\\
 &\qquad =   \sum_{B_{\ell+1}\in \mathcal B_{\ell+1}^\sigma} \sum_{B_\ell \in \mathcal B_{\ell}}
\l T_\sigma ^{s_{B_{\ell+1}}} [f_1 \phi_{B_\ell}  \phi_{B_{\ell+1}}],f_2  \rangle =
\sum_{B_{\ell+1}\in \mathcal B_{\ell+1}^\sigma} 
\l T_\sigma ^{s_{B_{\ell+1}}} [f_1 \ind_{E_\ell}  \phi_{B_{\ell+1}}],f_2  \rangle  
\\
&\qquad = \sum_{B_{\ell+1}\in \mathcal B_{\ell+1}^\sigma} 
\l T_\sigma ^{s_{B_{\ell+1}}}[ f_1   \phi_{B_{\ell+1}}],f_2  \rangle 
\end{split}
\]
which completes the inductive step. To pass to the second line we used the fact that whenever $B_{\ell+1}\in \mathcal B_{\ell+1}^\sigma$ and  $B_\ell\in \mathcal B_\ell$ is such that   $\phi_{B_\ell} \phi_{B_{\ell+1}}\neq 0$, then $r_{B_\ell}\geq2 r_{B_{\ell+1}}$, whence $B_\ell\in \mathcal B_\ell^\sigma$: in other words, $\phi_{B_\ell} \phi_{B_{\ell+1}}= 0$ if $B_\ell \not \in \mathcal B_\ell^\sigma$ . In the final equality we have used that $B_{\ell+1}\subset E_\ell$ for all $B_{\ell+1} \in \mathcal B_{\ell+1}$.
\end{proof}
%%%%%%%%%%%%%%%%%%%%%%%%%%%%%% PROOF PROOF PROOF

%%%%%%%%%%%%%%%%%%%%%%%%%%%%%% SECTION SECTION SECTION
\subsection{Proof of Lemma \ref{lemma:iter}} All the implicit constants throughout this section may depend on the parameters of the space of homogeneous type and on the constant $c_o$ of \eqref{eq:localization}. In order to clean up the presentation, the dependence on these parameters is hidden by the almost inequality sign.  
%%%%%%%%%%%%%%%%%%%%%%%%%%%%%% REMARK REMARK REMARK
\begin{remark}[Supports and scaling] \label{r:ss}
By definition of the stopping form,  we are free to assume $\mathrm{supp}\,h_1\subset L$. Because of the localization assumption \eqref{eq:localization} for $T$ and the fact that $\mathrm{supp}\,h_1\subset L$ and that the largest scale in the sum in \eqref{e:stofo} is $s_L$, we may assume that $\supp h_2 \subset c_o L$.
Furthermore, due to the presence of the partition of unity $\phi_B$ in the summation appearing in \eqref{e:stofo} we are allowed to purge from $\mathcal B$ the balls $B$ with $B\cap L= \varnothing$. In particular may assume $B\subset 2c_{\d} c_o L\subset qL$ for all $B\in\mathcal B$ since $s_B\leq s_L/2$ for all balls in the sum. 
\end{remark}
%%%%%%%%%%%%%%%%%%%%%%%%%%%%%% REMARK REMARK REMARK
For convenience, the stopping norms are normalized
\begin{equation}\label{eq:normalized}
\|h_1\|_{p_1,(L,\mathcal B)}= \|h_2\|_{p_2,(L,\mathcal B)}=1.
\end{equation}
We then need to bound $\tilde\Lambda(h_1,h_2)$, where  $\tilde\Lambda$ is the  form
\begin{equation}
\label{e:locl}
\tilde \Lambda(u_1,u_2)\coloneqq \left\langle T_\sigma ^{s_{L}}\left[ u_1\ind_{  E^c}\right], u_2 \right \rangle+ \sum_{B\in \mathcal B} \left\langle T^{s_L}_{{s_B\vee \sigma}} (u_1\phi_B), u_2 \right\rangle. 
\end{equation}
The main tools are the following   Calder\'on-Zygmund decompositions $h_i = b_i + g_i$ where 
\begin{equation}
\label{e:czdec1}
\begin{split}
 &b_i\coloneqq\sum_{B\in \mathcal B} b_{i,B}, \quad  b_{i,B}\coloneqq h_i\phi_B -  \left(\textstyle  {\intav_{\!B} h_i\phi_B } \right)\ind_{B},
\quad g_i=h_i\ind_{ E^c} + \sum_{B\in \mathcal B}  \left( \textstyle {\intav_{\!B} h_i\phi_B }\right)\ind_{B},
\end{split}
\end{equation}
The following equalities and  estimates are readily obtained from the definition of the stopping norms and \eqref{eq:normalized}:
\begin{align}
\label{e:supportz}&  \mathrm{supp}\, b_{i,B} \subset B, \qquad
\mathrm{supp}\, b_i,\,  \mathrm{supp}\, g_i \subset qL, \qquad B\in \mathcal B,\, i=1,2;
\\& \label{e:meanzero1}
 \int b_{i,B} =0  \qquad\forall B\in \mathcal B, \qquad i=1,2
\\
&\label{e.bad}\l b_i \r_{p_i,qL}\lesssim \sup_{B\in \mathcal B}\langle b_{i,B}\rangle_{p_i,B}  \lesssim 1, \qquad i=1,2;
\\ & \label{e.good}
\|g_i\|_{\infty} \lesssim 1, \qquad i=1,2.
\end{align}
We then decompose
\begin{equation}
\label{eq:split}
\begin{split}
\tilde \Lambda(h_1,h_2) =\tilde \Lambda(g_1,g_2) + \tilde \Lambda(b_1,g_2)  +\tilde \Lambda(g_1,b_2)+ \tilde \Lambda(b_1,b_2)  
 %\eqqcolon\mathrm{GG}+\mathrm{BG}+\mathrm{GB}+\mathrm{BB}
\end{split}
\end{equation}
and we proceed to estimate each one of these terms.
\subsubsection{Estimate for $\tilde \Lambda(g_1,g_2)$} This  is the easiest one. 
Using the $L^\infty$ estimate \eqref{e.good} and the packing condition, and appealing to the uniform $L^p$-bound on truncations,
\begin{equation}
\label{e:minitrick}
\begin{split}&\quad \left| H\coloneqq \sum_{B\in \mathcal B} \left\langle T^{{s_B\vee \sigma}}_{\sigma}  (g_1\phi_B), g_2 \right\rangle \right|  =
\left|  \sum_{B\in \mathcal B} \left \langle T^{{s_B\vee \sigma}}_{\sigma}  (g_1\phi_B), g_2 \ind_{c_o B}\right\rangle \right| 
\\ &\lesssim  C_{p}     \sum_{B\in \mathcal B} |B| \langle g_1\phi_B\rangle_{p,B} \langle g_2\rangle_{p',c_oB}   
 \lesssim C_{p} |L|.
\end{split}
\end{equation}
Using the partition of unity and  \eqref{e.good} again
\[
\begin{split}
|\tilde \Lambda(g_1,g_2) -H|=\left| 
 \left\langle T^{s_L}_\sigma \left( g_1\ind_{E^c}+\sum_{B\in \mathcal B} g_{1}\phi_B \right), g_2 \right\rangle\right|  = \left|
 \left\langle T^{s_L}_\sigma   g_1 , g_2 \right\rangle\right| 
 \lesssim C_{p}|L|  \|g_1\|_{\infty} \|g_2\|_{\infty}   \lesssim C_{p}|L|,
\end{split}
\]
which completes the bound for $\tilde \Lambda(g_1,g_2) $.
\subsubsection{The estimates for  $\tilde \Lambda(b_1,g_2) ,\tilde \Lambda(b_1,b_2) $} This argument is reminiscent of \cite[Lemma 4.2]{DPHL}. Let $ B,B'\in \mathcal B$ throughout.  Our first observation is that 
\begin{equation}
\label{e:switchscale} T^{s_L}_{s_B\vee \sigma} = T^{s_L}_{s_{B'}\vee \sigma} + \sign (s_{B'}-s_B) T^{s_B\vee s_{B'}}_{(s_B\wedge s_{B'})\vee \sigma}
\end{equation}
which is verified by tedious case by case analysis, keeping in mind that $T^{u}_v=0$ if $v\geq u$. The support condition \eqref{e:supportz} also entails 
\[
T^{s_L}_{s_B\vee \sigma} [b_{1,B'}\phi_B] \not\equiv 0\implies B\in N_{\mathsf{lft}}(B')=\{B\in \mathcal B:B\cap B'\neq\varnothing \}
\]
Therefore, also noticing that $b_{1,B'}\ind_{E^c}=0$, we can write for any function $u$
\begin{equation} \label{e:b1g21}
\begin{split}
& \tilde \Lambda(b_{1,B'},u)= \sum_{B\in \mathcal B} \left\langle  T^{s_L}_{s_B\vee \sigma} [b_{1,B'}\phi_B], u\right\rangle
=\sum_{B\in N_{\mathsf{lft}}(B')} \left\langle  T^{s_L}_{s_B\vee \sigma} [b_{1,B'}\phi_B], u\right\rangle
\\ 
&\qquad = \sum_{B\in N_{\mathsf{lft}}(B')} \left\langle  T^{s_L}_{s_{B'}\vee \sigma} [b_{1,B'}\phi_B], u\right\rangle + 
\sum_{B\in N_{\mathsf{lft}}(B')}  \sign (s_{B'}-s_B) \left\langle  T^{s_B\vee s_{B'}}_{(s_B\wedge s_{B'})\vee \sigma}[b_{1,B'}\phi_B], u\right\rangle 
\\ 
&\qquad =  \left\langle  T^{s_L}_{s_{B'}\vee \sigma} b_{1,B'}, u\right\rangle + 
\sum_{B\in N_{\mathsf{lft}}(B')}  \sign (s_{B'}-s_B) \left\langle  T^{s_B\vee s_{B'}}_{(s_B\wedge s_{B'})\vee \sigma}[b_{1,B'}\phi_B], u\right\rangle .
	\end{split}
\end{equation}
We estimate each of the summands in \eqref{e:b1g21}. The easiest is the first one. For $s\geq s_{B'}$, let $B'(s)$ be the ball with center same as $B'$ and radius $2^s$. Using localization, the adjoint $L^{p_1}$-improving property, and relying on the cancellation \eqref{e:meanzero1},
\begin{equation}\label{e:b1g22} 
\begin{split}
&\left|\left\langle  T^{s_L}_{s_{B'}\vee \sigma} b_{1,B'}, u\right\rangle\right| \leq \sum_{s_{B'}\vee \sigma\leq s<s_L}
\left|\left\langle  T(s) b_{1,B'}, u \ind_{c_oB'(s)}\right\rangle\right| 
\\ 
&\qquad \lesssim \sum_{s_{B'}\vee \sigma\leq s<s_L} \omega(2^{s_{B'}-s}) |B'(s)| \l  b_{1,B'} \r_{p_1,B'(s)} \l  u\r_{p_2,B'(s)} 
 \\
&\qquad =|B'| \l b_{1,B'}\r_{p_1,B'}\sum_{s_{B'}\vee \sigma\leq s<s_L} \omega(2^{s_{B'}-s})    \l  u \r_{p_2,B'(s)}
\\
&\qquad \lesssim \|\omega\|_{\mathrm{Dini}} |B'| \l b_{1,B'} \r_{p_1,B'}  \sup_{s_{B'}\vee \sigma\leq s<s_L}\l u \r_{p_2,B'(s)}.
\end{split}\end{equation}

There are $\#N_{\mathsf{lft}}(B')\lesssim 1$ terms involving $B\in N_{\mathsf{lft}}(B')$, by the Whitney property, and each is  handled as follows. Notice that  $2^{s_{B}}\sim2^{s_{B'}}$ if $B\in  N_{\mathsf{lft}}(B')$ by the Whitney property. Set
$
B^\pm= \arg\max\{\pm s_U:U\in \{B, B'\}\}$. To fix the ambiguity  when $s_B=s_{B'}$, we set $B^+=B', B^-=B$; this last choice is immaterial for the argument.
{Applying the single scale $L^{p_1}$-improving property \eqref{eq:Lpimpss}} and observing that there are at most $\lesssim 1$ scales between $s_{B^-}\vee \sigma$ and $ s_{B^+}$, 
\begin{equation}
\label{e:b1g23} 
\begin{split}
& \left| \left\langle  T^{s_B\vee s_{B'}}_{(s_B\wedge s_{B'})\vee \sigma}[b_{1,B'}\phi_B], u\right\rangle\right| \leq 
\sum_{s_{B^-} \vee \sigma\leq s<s_{B^+}} \left|\left\langle  T(s) [b_{1,B'}\phi_B], u \ind_{c_oB^+}\right\rangle\right| 
\\ 
&\qquad \lesssim   |B^+| \l  b_{1,B'}\phi_B \r _{p_1,B^+} \l  u \r_{p_2,c_oB^+} \lesssim |B'|\l  b_{1,B'} \r_{p_1,B'}  \sup_{s_{B'}\vee \sigma\leq s<s_{\mathbb X}+s_L}\l u \r_{p_2, c_oB'(s)}
\end{split}
\end{equation}
where $s_{\mathbb X}\geq 0$ is a constant depending on the homogeneous structure of $\,\mathbb X$; the last inequality follows from the fact that $|B'|\sim |B^+|$ and there is $s\sim s_{B^+}$ such that $B^+\subset c_oB'(s)$. Collecting \eqref{e:b1g21}, \eqref{e:b1g22} and \eqref{e:b1g23}, it follows
\[
|\tilde \Lambda(b_{1,B'},u)| \lesssim \|\omega\|_{\mathrm{Dini}}  |B'|\l b_{1,B'}\r_{p_1,B'} \left( \sup_{s_{B'}\vee \sigma\leq s<s_{\mathbb X}+s_L}\l u\r_{p_2, B'(s)}\right).
\]
We plan to apply the estimate above for functions $u$ supported in $c_oL$. Then a further simplification can be made  by noticing that
\[
\sup_{ s_L \leq s<s_{\mathbb X}+s_L}\l u\r_{p_2, B'(s)} \leq \sup_{\substack{ s_L \leq s<s_{\mathbb X}+s_L\\ B'(s)\cap c_oL\neq \varnothing}} \frac{|c_oL|}{|B'(s)|}\l u \r_{p_2,c_oL}\lesssim_\eps \l u \r_{p_2,c_oL}
\]
Using this we can rewrite the estimate above in the form
\begin{equation}\label{eq:badmax}
|\tilde \Lambda(b_{1,B'},u)| \lesssim \|\omega\|_{\mathrm{Dini}}  |B'|\l b_{1,B'}\r_{p_1,B'}  \max \left( \sup_{s_{B'}\vee \sigma\leq s< s_L}\l u\r_{p_2, c_oB'(s)},\l u_2\r_{c_oL}\right).
\end{equation}
We now bound $\Lambda(b_{1},g_2)$ immediately. Indeed, using \eqref{e.bad}, \eqref{e.good}
\begin{equation}\label{eq:bg}
\begin{split}
|\tilde \Lambda(b_{1},g_2)| & \leq \sum_{B'\in \mathcal B}  |\tilde \Lambda(b_{1,B'},g_2)| \lesssim \|\omega\|_{\mathrm{Dini}}
\sum_{B'\in \mathcal B}  |B'|\l b_{1,B'} \r_{p_1,B'} \| g_2 \|_{\infty}
\\
&\lesssim  \|\omega\|_{\mathrm{Dini}}
\sum_{B'\in \mathcal B}  |B'|\lesssim \|\omega\|_{\mathrm{Dini}}|L|.\end{split}
\end{equation}
We then bound $\tilde\Lambda(b_{1},b_2)$. Preliminarily notice that $\l b_2 \r_{p_2,c_o L}\lesssim \l b_2 \r_{p_2,qL}\lesssim 1$ by \eqref{e.bad}. Thus we focus on estimating $\l b_2\r_{p_2,c_oB'(s)}$ for  $s_{B'}\vee \sigma\leq s<s_L$ and fixed $B'\in\mathcal B$. Consider all $B\in \mathcal B$ such that $B\cap c_oB'(s)\neq \emptyset$. Then either $B'(s)\subseteq qB $,  or $B'(s)\nsubseteq  q B $.  In the first case we say that $(s,B)$ is of type 1 otherwise we say it is of type 2. If $(s,B)$ is of type 1, then  $B'\subseteq B'(s)\subseteq qB$ which occurs for $O_{\mathbb X}( 1)$ balls $ B\in \mathcal B$ since $\mathcal B$ is a $q$-Whitney decomposition of some open set. For type 1 pairs $(s,B)$ we then have by property (iv) of Lemma~\ref{lem:whitney} and the fact that $\mathcal B$ is $q$-Whitney, that $2^{s_{B'}}\sim 2^s \sim 2^{s_B}$ and therefore, using \eqref{e.bad} in the last step,
\[
\begin{split}
\langle| b_2|\rangle_{p_2,c_o B'(s)} \leq \sum_{ \substack {B\in  \mathcal B\\    qB\supseteq B'(s)\supseteq B'  }}\langle b_{2,B} \rangle_{p_2,c_o B'(s)} \lesssim 1.
\end{split}
\]
If $(s,B)$ is of type 2 instead the by the definition of $q$ we necessarily have that $B\subseteq q B'(s)$; following up the finite overlap of $\mathcal B$ with \eqref{e.bad} and then using the packing condition, 
\[
\begin{split}
\langle| b_2|\rangle_{p_2, B'(s)}^{p_2} \leq  \sum_{ \substack {B\in  \mathcal B\\   B\subseteq  q B'(s) }} \frac{|B|}{|B'(s)|}\langle| b_2|\rangle_{p_2,B}^{p_2}  \lesssim 1.
\end{split}
\]
It follows that, using \eqref{eq:badmax} and the estimates above, that
\begin{equation}
\label{e:b1b2fin}
\begin{split}
|\tilde \Lambda(b_{1},b_2)| &  \leq \sum_{B'\in \mathcal B}  |\tilde \Lambda(b_{1,B'},b_2)| 
\\
& \lesssim \|\omega\|_{\mathrm{Dini}}
\sum_{B'\in \mathcal B}  |B'|\langle b_{1,B'}\rangle_{p_1,B'} \max\left( \sup_{s_{B'}\vee \sigma\leq s<s_L}\langle b_2\rangle_{p_2,c_o B'(s)},\l u\r_{p_2,c_oL}\right)
\\
&\lesssim  \|\omega\|_{\mathrm{Dini}}
\sum_{B'\in \mathcal B}  |B'|\lesssim \|\omega\|_{\mathrm{Dini}}|L|.\end{split}
\end{equation}

\subsubsection{The estimate  $\tilde \Lambda(g_1,b_2) $} This is similar to the previous bounds but uses the adjoint $U^v_u=(T_u^v)^*$. Let $B'\in \mathcal B$. Now $U^v_u=(T_u^v)^*$ satisfies the same assumptions as $ T_u^v$ and we may reuse the material of the previous subsection.  First of all, if $s_B>\sigma$, $U_{ \sigma}^{s_B} b_{2,B'} $ vanishes off $E$, because $c_oB'\subset E$ since $c_o<q$ and $B'\in\mathcal B$ is a $q$-Whtiney cover of $E$. Therefore
\begin{equation}
\label{e:g1b21}
|\tilde
\Lambda(g_1\ind_{E^c}, b_{2,B'})| = |\langle g_1\ind_{E^c}, U_{ s_{B'}}^{s_L}b_{2,B'} \rangle | \lesssim \|\omega\|_{\mathrm{Dini}} |B'|
\end{equation}
which is dual to an instance of \eqref{eq:bg}. Then, defining $N_{\mathsf{rgh}}(B')=\{B\in \mathcal B:B\cap c_oB'\neq \varnothing\}$, using the support condition and the equality \eqref{e:switchscale} for $U $ in place of $ T$, 
\[
\Lambda(g_1\ind_{E}, b_{2,B'}) =  \langle g_1\ind_{E}, U_{ s_{B'}}^{s_L}b_{2,B'} \rangle + 
\sum_{B\in N_{\mathsf{rgh}}(B') } \sign (s_{B'}-s_B) \left\langle g_{1}\phi_B,  U^{s_B\vee s_{B'}}_{(s_B\wedge s_{B'})\vee \sigma} b_{2,B'}\right\rangle.
\]
Applying, with $U $ in place of $ T$, \eqref{e:b1g22} for the first term  and  \eqref{e:b1g23} for the $\#N_{\mathsf{rgh}}(B')\lesssim 1$ summands in the second term, we obtain
\begin{equation}
\label{e:g1b22}
\left|\tilde
\Lambda(g_1\ind_{E}, b_{2,B'})\right|    \lesssim \|\omega\|_{\mathrm{Dini}} |B'|.
\end{equation}
Combining \eqref{e:g1b21} with \eqref{e:g1b22} and summing over $B'\in \mathcal B$ yields the sought after estimate for $\tilde \Lambda(g_1,b_2) $ and completes the proof of the Lemma.
%%%%%%%%%%%%%%%%%%%%%%%%%%%%%% SECTION SECTION SECTION
\section{Proof of Theorem~\ref{th:max}}\label{sec:max}  In this last section we provide the proof of the maximal version of our sparse domination result as formulated in Theorem~\ref{th:max}. The only essential difference with the proof of Theorem~\ref{th:SD} is the stopping form to be estimated involves maximal averages. Let $\{T(s)\}_{s\in\Z}$ be a sequence of linear operators satisfying the assumptions of Theorem~\ref{th:max}, in particular \eqref{eq:localmax}. Redefine for $f\in L^\infty(\mathbb X)$
\begin{equation}
\label{e:Ms}
\mathrm{M}_\sigma^\tau f(x) = \sup_{\sigma \leq s <\tau} \left|[T(s) f](x)\right|, \qquad x\in \mathbb X.
\end{equation} 
As  $\mathrm{M}_\sigma^\tau f\leq T_\star f$, by assumption $\|\mathrm{M}_\sigma^\tau\|_{L^\infty(\mathbb X)}\leq \|T_\star\|_{L^\infty(\mathbb X)}$ uniformly in $\sigma, \tau$.
Let $L,\mathcal B$ be as in Subsection \ref{ss:sf} and partition of unity $\{\phi_B:B\in \mathcal B\}$ subordinate to $\mathcal B$. The stopping form  of  \eqref{e:Ms} associated to this data is\begin{equation}
\label{e:stofom}
\Lambda^{\sigma,\tau} _{(L,\mathcal B)}(h_1,h_2)\coloneqq  \left\langle \mathrm{M}_\sigma^{s_L \wedge \tau}\left[ h_1\ind_{L\setminus E}\right],  h_2 \right\rangle  +\sum_{B\in\mathcal B}\left\langle  \mathrm{M}^{s_L \wedge \tau}_{ s_{B}\vee \sigma} \left[h_1 \ind_{L} \phi_{B}\right], h_2\right\r.
\end{equation}
Theorem~\ref{th:max} then follows from the next lemma by repetition of the same procedure that was carried out in Subsection \ref{s.together}.
%%%%%%%%%%%%%%%%%%%%%%%%%%%%%% LEMMA LEMMA LEMMA
\begin{lemma}\label{lemma:iter2} Let  $T$ satisfy the assumptions of Theorem~\ref{th:max} and $L,\mathcal B$ be as in \eqref{eq:stopwhit}. Then
\[
\sup_{\sigma>0}\left|\Lambda^\sigma _{(L,\mathcal B)}(h_1,h_2)\right| \lesssim_{X} \left[C_{ p }+  \|\omega\|_{\mathrm{Dini}}\right]|L| \|h_1\|_{p_1,(L,\mathcal B)} \|h_2\|_{p_2,(L,\mathcal B)}.
\]
\end{lemma}
\begin{proof}
We repeat the same reductions operated at the beginning of  the proof of Lemma \ref{lemma:iter}, in particular, those performed in Remark \ref{r:ss}.
In particular 
\[ \mathrm{supp}\,h_1\subset L, \qquad    \mathrm{supp}\,h_2\subset c_o L\subset q L, \qquad
\|h_1\|_{p_1,(L,\mathcal B)}= \|h_2\|_{p_2,(L,\mathcal B)}=1.
\]
However, in addition, we may assume that $h_2\geq 0$ without loss of generality. We repeat for the Calder\'on-Zygmund decomposition   $h_1 = b_1 + g_1$ of  
\eqref{e:czdec1}, and in particular \eqref{e:supportz} to \eqref{e.good} hold for $i=1$. For $h_2$   no such decomposition is necessary. In analogy with \eqref{e:locl} we write 
\[
\tilde \Lambda(u_1,u_2)= \left \langle
 \mathrm{M}_\sigma^{s_L \wedge \tau}\left[ u_1\ind_{ E^c}\right],  u_2 \right\rangle 
  +\sum_{B\in\mathcal B}\left\langle  \mathrm{M}^{s_L \wedge \tau}_{ s_{B}\vee \sigma} \left[u_1 \ind_{L} \phi_{B}\right], u_2\right\r 
\] and proceed with estimating the two terms  $ \tilde \Lambda(g_1,h_2), \tilde   \Lambda(b_1,h_2)$.
\subsubsection*{The estimate for $\tilde \Lambda (g_1,h_2)$} This term is rather easy.
By the support conditions \eqref{e:supportz} and an appeal to the $\|T_\star\|_{L^\infty(\mathbb X)}$ estimate, followed by the $L^\infty$ estimate \eqref{e.good} for $i=1$, the stopping norm estimate for $h_2$ and the packing condition, 
\[
\begin{split}&\quad \left| H\coloneqq \sum_{B\in \mathcal B}
 \left\langle \mathrm{M}^{s_B \wedge \tau}_{\sigma}  (g_1\phi_B), h_2 \right\rangle \right|  =
\left|  \sum_{B\in \mathcal B} \left\langle \mathrm{M}^{s_B \wedge \tau}_{\sigma}   (g_1\phi_B), h_2 \ind_{c_o B}\right\rangle \right| 
\\ &\lesssim  \|T_\star\|_{L^\infty(\mathbb X)} \|g_1\|_\infty    \sum_{B\in \mathcal B} |B|  \langle h_2 \rangle_{p_2,c_oB}   
 \lesssim C_{p} |L|.
\end{split}
\]
Relying on the partition of unity and  \eqref{e.good} again
\[
\begin{split}
|\tilde \Lambda(g_1,h_2) -H|=\left| 
 \left\langle  \mathrm{M}^{s_L \wedge \tau}_{\sigma}  g_1 , h_2 \right\rangle\right| 
 \lesssim \|T_\star\|_{L^\infty(\mathbb X)}|L|  \|g_1\|_{\infty} \|h_2\|_{p_2, aL }   \lesssim \|T_\star\|_{L^\infty(\mathbb X)}|L|,
\end{split}
\]
thus achieving the correct estimate for  $\tilde \Lambda(h_1,g_2) $.
\subsubsection*{The estimate for  $\tilde \Lambda(b_1,h_2)  $}  In analogy with \eqref{e:b1g21}, for $u\in L^\infty(\mathbb X)$ and nonnegative 
\begin{equation}
\label{e:b1g21m}
\begin{split}
&\quad  \tilde \Lambda(b_{1,B'},u) \leq    \left\langle   \mathrm{M}^{s_L \wedge \tau}_{s_{B'}\vee \sigma} b_{1,B'}, u\right\rangle + 
\sum_{B\in N_{\mathsf{lft}}(B')}   \left\langle  \mathrm{M}^{(s_B\vee s_{B'})\wedge \tau}_{(s_B\wedge s_{B'})\vee \sigma}[b_{1,B'}\phi_B], u\right\rangle.
\end{split}
\end{equation}
Using localization, we decompose the first summand \eqref{e:b1g21m} by splitting
\[
 \left\langle \mathrm{M}^{s_L \wedge \tau}_{s_{B'}\vee \sigma} b_{1,B'}, u\right\rangle =  \sum_{s_{B'}\vee \sigma\leq s<s_L\wedge \tau}
\left\langle  b_{1,B'},  T^*(s) \left[u\eps_s \ind_{c_oB'(s)}\right]\right\rangle 
 \]
 where $B'(s)$ is again the ball with center same as $B'$ and radius $2^s$  and 
\[
A_s\coloneqq \left\{x\in \mathbb X: \mathrm{M}^{s_L \wedge \tau}_{s_{B'}\vee \sigma} b_{1,B'} (x) = |[T(s) f](x)| \right\},\quad \eps_s\coloneqq \ind_{A_s} \sign [T(s) f], \quad s_{B'}\vee \sigma \leq s <s_L \wedge \tau.
\]
An appeal to the   $(p_2, p_1')$-improving property of $T^*(s)$, which is possible in view of the cancellation \eqref{e:meanzero1}, entails

\begin{equation}\label{e:b1g22m} 
\begin{split}
&\quad\sum_{s_{B'}\vee \sigma\leq s<s_L\wedge \tau}
\left\langle  b_{1,B'},  T^*(s) \left[u\eps_s \ind_{c_oB'(s)}\right]\right\rangle \\ &\leq  |B' | \langle b_{1,B'}\rangle_{p_1,B'} \sum_{s_{B'}\vee \sigma\leq s<s_L\wedge \tau} \omega(2^{s_{B'}-s})\langle u\eps_s \rangle_{p_2, c_oB'(s)}  
 \lesssim \|\omega\|_{\mathrm{Dini}} |B'|   \sup_{s_{B'}\vee \sigma\leq s<s_L}\langle u\rangle_{p_2,c_o B'(s)}.
\end{split}
\end{equation}

Similarly, for the  $\#N_{\mathsf{lft}}(B')\lesssim 1$ terms involving $B\in N_{\mathsf{lft}}(B')$,  this time applying  the single scale $(p_2,p_1')$-improving property, and observing that there are at most $\lesssim 1$ scales between $s_{B^-}\vee \sigma$ and $ s_{B^+}\wedge \tau$, 
\begin{equation}
\label{e:b1g23m} 
\begin{split}
 \left\langle  \mathrm{M}^{(s_B\vee s_{B'})\wedge \tau}_{(s_B\wedge s_{B'})\vee \sigma}[b_{1,B'}\phi_B], u\right\rangle \lesssim I^* |B'|\langle| b_{1,B'}|\rangle_{p_1,B'}  \sup_{s_{B'}\vee \sigma\leq s<s_L}\langle u\rangle_{p_2,c_o B'(s)}.
\end{split}
\end{equation}
Bringing \eqref{e:b1g21m}, \eqref{e:b1g22m} and \eqref{e:b1g23m} together,  we have proved that
\[
\tilde \Lambda(b_{1,B'},u) \lesssim \|\omega\|_{\mathrm{Dini}}  |B'|  \left( \sup_{s_{B'}\vee \sigma\leq s<s_L}\langle u\rangle_{p_2,c_o B'(s)}\right).
\]
Via the same argument applied towards the proof of \eqref{e:b1b2fin} we realize that 
\[
  \sup_{s_{B'}\vee \sigma\leq s<s_L}\langle h_2\rangle_{p_2,c_o B'(s)} \lesssim \|h_2\|_{p_2,(L,\mathcal B) } =1,
\]
so that the correct estimate for $\tilde \Lambda(b_1,h_2) $ follows by plugging the last display into \eqref{e:b1g23m} and summing over $B'\in \mathcal B$. This completes the proof of the lemma.
\end{proof}
%%%%%%%%%%%%%%%%%%%%%%%%%%%%%% LEMMA LEMMA LEMMA

\bibliography{MetricSD}
\bibliographystyle{amsplain}
\end{document}